\documentclass[10pt]{article}
\usepackage[utf8]{inputenc}
\usepackage{geometry}
\usepackage{amsmath}
\usepackage{amsfonts}
\usepackage{mathrsfs}
\usepackage{amsthm}
\usepackage{amssymb}
\usepackage{hyperref}
\usepackage{lineno}
\usepackage[capitalize]{cleveref}
\usepackage{extarrows}
\usepackage{enumitem}
\usepackage{bbm}
\usepackage{upgreek}
\usepackage{mathtools}
\usepackage{wasysym}
\usepackage{xfrac}
\usepackage{graphicx}
\usepackage{subfig}
\usepackage{float}
\usepackage{multirow}

\theoremstyle{plain}
\newtheorem{theorem}{Theorem}[section]
\newtheorem{lemma}[theorem]{Lemma}

\newtheorem{corollary}[theorem]{Corollary}

\theoremstyle{definition}
\newtheorem{definition}{Definition}[section]

\theoremstyle{remark}
\newtheorem*{remark}{Remark}

\usepackage[nottoc]{tocbibind}

\usepackage{authblk}

\renewcommand{\d}[1]{\ensuremath{\operatorname{d}\!{#1}}}

\title{Arbitrarily High-order Trapezoidal Rules for Functions with Fractional Singularities in Two Dimensions}
\author[1]{Senbao Jiang}
\author[1]{Xiaofan Li}
\affil[1]{Department of Applied Mathematics, Illinois Institute of Technology}

\date{\today}

\begin{document}

\maketitle

\begin{abstract}
  In this paper, we introduce and analyze arbitrarily high-order quadrature rules for evaluating the two-dimensional singular integrals of the forms \begin{align*}
      I_{i,j} = \int_{\mathbb{R}^2}\phi(x)\frac{x_ix_j}{|x|^{2+\alpha}} \d x, \quad 0< \alpha < 2
  \end{align*} where $i,j\in\{1,2\}$ and $\phi\in C_c^N$ for $N\geq 2$. This type of singular integrals and its quadrature rule appear in the numerical discretization of fractional Laplacian in non-local Fokker-Planck Equations in 2D. The quadrature rules are trapezoidal rules equipped with correction weights for points around singularity. We prove the order of convergence is $2p+4-\alpha$, where $p\in\mathbb{N}_{0}$ is associated with total number of correction weights. Although we work in 2D setting, we formulate definitions and theorems in $n\in\mathbb{N}$ dimensions when appropriate for the sake of generality. We present numerical experiments to validate the order of convergence of the proposed modified quadrature rules.
\end{abstract}

%%%%%%%%%%%%%%%%%%%%%%%%%%%%%%%%%%%%%%%%%%%%%%%%%%%%%%%%%%%%%%%%%%%%%%%%%%%%%%%%%%%%%%%%%%%%%%%%%%%%%%%%%%%%%%%%%%%%%%%%%%%%%%%%%%%%%%%%%%%%
\section{Introduction}
In this article we design and analyze arbitrarily high-order quadrature rules for a class of weakly singular integrals in two dimensions, with an isolated singularity at the origin. Such weakly singular integrals arises naturally in non-local partial differential equations involving the fractional Laplacian. 

There are different methods of numerical integration for weakly singular integrals, such as singularity subtraction and singularity removing by change of co-ordinates, see for example \cite{10.2307/2157208} \cite{Bruno2001AFH} \cite{Atkinson2004QUADRATUREOS} \cite{1514571} \cite{Sidi2005ApplicationOC} \cite{Mousavi2009GeneralizedDT}. The approach we use here is to modify the trapezoidal rule around singularity in order to achieve high-order convergence rate. This method is efficient as it introduces correction weights for the trapezoidal rule near the singularity. The correction weights are pre-computed. Once being stored, the modified trapezoidal rule can be easily set up. It is well-known that the trapezoidal rule is spectrally accurate for smooth functions with compact support or periodic functions, while it is second order accurate for other functions. There are two main sources of errors: one comes from singularities and the other is from the boundaries. We focus on the error caused by singularity in this paper.

Rokhlin \cite{ROKHLIN199051} was first to propose singularity-corrected trapezoidal rule and Alpert \cite{Alpert1990RapidlyConvergentQF}, Kapur and Rokhlin \cite{doi:10.1137/S0036142995287847} further improves the method. Their quadrature rules are designed for functions of the form $f(x) = \phi(x) s(x) + \Psi(x)$ or $f(x) = \phi(x) s(x)$ in 1D, where $\phi(x),\Psi(x)$ are regular functions and $s(x)$ is a singular function with isolated singularity such as $s(x) = |x|^{\gamma},\gamma>-1$ or $s(x) = \log(x)$. Aguilar and Chen \cite{AGUILAR20021031, AGUILAR2005625} designed singularity-boundary-corrected trapezoidal rule for singular kernel $s(x) = \log(x)$ in 2D and $s(x) = \sfrac{1}{|x|}$ in 3D.

The method we used for constructing increasingly accurate modified trapezoidal rule is to require the modified rule exactly evaluate the integral of the product of some monomials with increasing degree and the singular kernel. This gives rise to a linear system of equations for the correction weights. The more weights one is using, the larger the linear system is. In general, the weights obtained from solving the linear systems depend on the mesh size $h$, as in Keast and Lyness \cite{Keast1979OnTS}, Duan and Rokhlin \cite{doi:10.1137/S0036142995287847}. In the latter paper, the authors are able to derive analytical formula for the weights. Analytic formula are not feasible for the weights relying on mesh size, especially when the number of weights becomes large. In studying weakly singular integral involving singular kernels of the forms $s(x) = |x|^\gamma,\gamma>-1$ in 1D and $s(x) = \sfrac{1}{|x|}$ in 2D, Marin \textit{et al} \cite{MarinTornberg2014} designed a modified trapezoidal rule and managed to obtain the weights independent of the mesh size, by letting $h\to 0$. Those weights can then be pre-computed and stored. They provided convergence analysis for the weights in 1D and the numerical evidence of the convergence in 2D. They also gave a rigorous proof for the rate of convergence of the quadrature rule in 1D and conditional proof for the two dimensional counterpart. In fact, the difficulty for proving the linear systems of the correction weights has a well-defined solution when $h\to 0$ lies in showing the linear systems are non-singular as the size of the systems varies with the number of the correction weights. The same authors also applied the quadrature rule to study the fundamental solution of the Stokes equations in \cite{MARIN2012215}. 

In \cite{HansenHa}, Ha investigates numerically non-local Fokker-Planck equation in two dimensions  \begin{align}
    u_t(x,t) & = -\partial_1(b_1(x)u(x,t)) - \partial_2(b_2(x)u(x,t)) + \frac{1}{2}\sigma_{11}^2\partial_1^2u(x,t) + \frac{1}{2}\sigma_{22}^2\partial_2^2u(x,t) \\
    & + C_{2,\alpha}\int_{\mathbb{R}^2\setminus\{0\}}\frac{u(x,t) - u(x+y,t)}{|y|^{2+\alpha}}\d y, \ (t,x)\in [0,1]\times \Omega\\
    u(x,t) & \equiv 0, \ x\not \in \Omega.
\end{align} Here, 
\begin{align}
(-\Delta)^{\frac{\alpha}{2}}u(x) &\coloneqq C_{2, \alpha}\int_{\mathbb{R}^2\setminus\{0\}}\frac{u(x) - u(x+y)}{|y|^{2+\alpha}}\d y, \quad \alpha\in(0,2), \\
\end{align} is the fractional Laplacian and \begin{align}
    C_{d,\alpha} &\coloneqq \frac{2^{\alpha}\Gamma(\frac{\alpha+d}{2})}{\pi^{\frac{d}{2}}|\Gamma(-\frac{\alpha}{2})|},\ d\in\mathbb{N}.
\end{align} Numerical discretization of fractional Laplacian boils down to designing accurate quadrature rules for the weakly singular integrals $I_{i,j} = \int_{\mathbb{R}^2}\phi(x) \frac{x_ix_j}{|x|^{2+\alpha}}\d x$ , where $\phi$ is a regular function and $i,j\in\{1,2\}$. Ha \cite{HansenHa} adapted the simplest form of the modified trapezoidal rule of Marin et al \cite{MarinTornberg2014} for the issue. The quadrature rule is effective in evaluating the fractional Laplacian, however there is no rigorous convergence analysis or comprehensive numerical investigation for the rule itself. For more details please refer to \cite{HansenHa}. For a comprehensive review of fractional Laplacian please refer to \cite{LISCHKE2020109009} and references therein.

In this work, we propose the complete version of the modified trapezoidal rule for the two dimensional weakly singular integrals of the forms:\begin{align} \label{singularint}
    I_{i,j} &= \int_{\mathbb{R}^2}\phi(x) \frac{x_ix_j}{|x|^{2+\alpha}}\d x,\quad \alpha\in(0,2).
\end{align} Here, $\phi(x)$ is a function with compact support satisfying some smoothness criteria and $i,j\in\{1,2\}$. We call $I_{11},I_{22}$ the \emph{on-diagonal} weakly singular integrals and $I_{12},I_{21}$ the \emph{off-diagonal} weakly singular integrals. We provide theoretical analysis of the proposed quadrature rules and verify them by numerical results. We prove the quadrature rule has the order of convergence $2p+4-\alpha,\ p\geq 0$ for on-diagonal case and $2p+2-\alpha,\ p\geq 2$ for off-diagonal case, where $p$ is associated with total number of correction weights. Our correction weights are independent with the mesh size. We are able to show the linear systems of the correction weights are non-singular for arbitrary $p$. By using our method one could give unconditional affirmation to the conditional proof of convergence rate in \cite{MarinTornberg2014}. In the proof of the order of convergence of our quadrature rule, we generalize the ideas used in the proofs in \cite{MarinTornberg2014}. 

We organize the paper as follows. In \cref{sectionCTR}, we introduce the modified trapezoidal rules for on/off-diagonal weakly singular integrals and their associated matrix formulations for determining correction weights. We prove some preliminary results and the main theorems in \cref{sectionTA}. We address the technical non-singularity problem arising from matrix formulation of section \cref{sectionCTR} in \cref{sectionNSCM}. In final \cref{sectionNR} numerical results for order of convergence and associated correction weights are presented for both cases. In our presentation, we elaborate on on-diagonal case and outline those off-diagonal counterparts.

%%%%%%%%%%%%%%%%%%%%%%%%%%%%%%%%%%%%%%%%%%%%%%%%%%%%%%%%%%%%%%%%%%%%%%%%%%%%%%%%%%%%%%%%%%%%%%%%%%%%%%%%%%%%%%%%%%%%%%%%%%%%%%%%%%%%%%%%%%%%%%%%%%%%%%%%%%%%%%%
\section{The Corrected Trapezoidal Rules} \label{sectionCTR}

\subsection{Notations}
Through-out this paper, we use $h$ for mesh size, lowercase letters such as $x,y,z$ for either scalar or vector and bold uppercase letters such as $\boldsymbol{A},\boldsymbol{B}$ for matrices. The Euclidean norms $|x|\coloneqq \sqrt{\sum_{i = 1}^n x_i^2}$ and $|x|_1:=\sum_{i=1}^n|x_i|$. The natural number with zero $\mathbb{N}_0\coloneqq \{0,1,2,\cdots\}$.
Multi-index notations appeared frequently in this paper. In particular, we use $p\leq q$ for $p = (p_1,\cdots,p_n),\ q = (q_1,\cdots,q_n)$ if and only if $p_i\leq q_i$ for all $i$ and $x^{\gamma} \coloneqq \prod_j x_j^{\gamma_j}$ for any multi-index $\gamma = (\gamma_1,\cdots,\gamma_n)$. 
\subsection{On-diagonal Corrected Trapezoidal Rule}
For any compactly supported function $f$ on $\mathbb{R}^n$, the punctured-hole trapezoidal rule is defined by \begin{align}\label{Punctured_Hole_Trapez}
    T_h^0[f]\coloneqq h^n\sum_{\beta\in\mathbb{Z}^n\setminus\{0\}} f(\beta h).
\end{align} Let $f=\phi\cdot s$, where $\phi\in C_c(\mathbb{R}^n)$ and $s$ is a singular function of the form \begin{align}\label{SingularFunctionGeneral}
    s(x)\coloneqq \frac{x_1^2}{|x|^{n+\alpha}}, \quad \forall x\in\mathbb{R}^n,\ \text{with}\ 0<\alpha<2.
\end{align} We have chosen $s$ to be associated with $x_1$, the results associated with this $s$ can be easily extended to other $x_i$'s. We call such $s$ on-diagonal singular kernel. Let $p\in\mathbb{N}_0$, we introduce a corrected trapezoidal rule in $\mathbb{R}^2$ \begin{align}\label{FirstCorrectedTrapezRule}
Q_h^p[f]\coloneqq T_h^0[f]+a(h)\sum_{\beta\in\mathcal{L}_p}\bar{\omega}_{\beta}\phi(\beta h),
\end{align} where $\bar{\omega}_{\beta}$'s are the correction weights to be defined, $\mathcal{L}_p = \{\beta\in\mathbb{Z}^2:|\beta|_1\leq p\}$. $a(h)$ is the leading error term between the singular integral and punctured-hole trapezoidal rule and it turns out to be $h^{2-\alpha}$. To see this, we can roughly evaluate \begin{align*}
    \int_{(-h,h)^2}|\phi(x)|s(x)\ \d x &\sim \int_{B(0,h)}|\phi(x)|s(x)\ \d x \\
    & \lesssim \int_{B(0,h)}s(x)\ \d x \lesssim \int_0^h r^{1-\alpha}\ \d r \sim h^{2-\alpha},
\end{align*} where $\sim$ means on the same order as $h\to 0$. In \cref{sectionTA} we will show rigorously that $a(h)=h^{2-\alpha}$.

We now elaborate on how to determine the correction weights $\bar{\omega}_{\beta}$. For each $h>0$ we require the corrected trapezoidal rule $Q_h^p$ in \cref{FirstCorrectedTrapezRule} with weights $\omega_{\beta}(h)$ evaluate the following integrals exactly \begin{align}\label{FirstCoeffLinearSystem}
    \int_{\mathbb{R}^2} g(x)s(x)x^{2\xi}\d x &=   T_h^0[g\cdot s\cdot x^{2\xi}]  \nonumber \\ 
     & +a(h)\sum_{\beta\in\mathcal{L}_p}\omega_{\beta}(h)(\beta h)^{2\xi}g(\beta h), \quad \forall \xi\in\mathbb{N}_0^2:|\xi|_1\leq p,
\end{align} where $g$ is a radially symmetric, smooth function with compact support such that $g(x) \not \equiv 0$. By symmetry of the integrand $g\cdot s$, we impose same symmetry on weights \begin{align}
    \bar{\omega}_{\beta_1,\beta_2} = \bar{\omega}_{|\beta_1|,|\beta_2|} , \quad \forall\beta = (\beta_1,\beta_2)\in \mathcal{L}_p.
\end{align} \Cref{fig:On_diag_grid} illustrates the grid of the correction weights for $p = 6$ after imposing symmetry. We say that the grid of the on-diagonal correction weights for $p = N, N\in \mathbb{N}_0$ is made up of $N+1$ correction layers, each layer is the set of points $\beta\in \mathbb{N}^2$ such that $|\beta|_1 = k,\ k = 0,\cdots,N$.
\begin{figure}[h]
    \centering
    \includegraphics[width = 10cm]{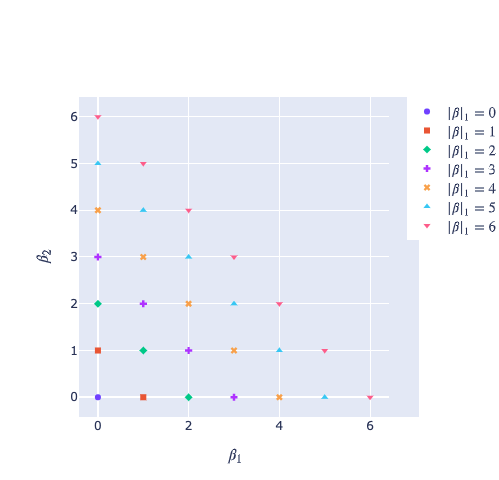}
    \caption{The location of the on-diagonal grid of correction weights after imposing symmetry when $p = 6$.}
    \label{fig:On_diag_grid}
\end{figure}

To facilitate further analysis, we define \begin{align}
    \mathcal{I}_p &\coloneqq \{(q-m,m):0\leq q\leq p,\ 0\leq m\leq q\}, \label{Ip_on} \\ 
    \mathcal{G}_{q-m}^m & \coloneqq \{(\pm (q-m),\pm m),(\mp (q-m),\pm m)\}, \quad \text{for each }(q-m,m)\in\mathcal{I}_p,\\
     \mathcal{G} & \coloneqq \{\mathcal{G}_{q-m}^m:(q-m,m)\in\mathcal{I}_p\}.
    \end{align} It is clear that $\mathcal{I}_p = \{\xi\in\mathbb{N}_0^2:|\xi|_1\leq p\}$ and $|\mathcal{G}| =|\mathcal{I}_p|=N_p$, where $N_p \coloneqq \frac{(p+1)(p+2)}{2}$. We occasionally use $\mathcal{G}_\gamma$ to denote $\mathcal{G}_{q-m}^m$ where $\gamma=(q-m,m)\in\mathcal{I}_p$ for the sake of notational convenience. We write the system of equations \cref{FirstCoeffLinearSystem} as $N_p\times N_p$ linear system for the weights $\omega_{\gamma}$ \begin{align} \label{RawLinearSystem}
        a(h)\sum_{\gamma\in\mathcal{I}_p}\omega_{\gamma}(h)\sum_{\beta\in\mathcal{G}_{\gamma}}(\beta h)^{2\xi}g(\beta h)= \int_{\mathbb{R}^2} g(x)s(x)x^{2\xi}\, \d x - T_h^0[g\cdot s\cdot x^{2\xi}], \quad \xi\in\mathcal{I}_p.
\end{align} and re-write \cref{FirstCorrectedTrapezRule} into \begin{align} \label{SecondCorrectedTrapezRule}
    Q_h^p[f] = T_h^0[f] + A_h^p[\phi]
\end{align}where \begin{align} \label{A^p}
    A^p_h[\phi]  = a(h)\sum_{q=0}^p\sum_{m=0}^{q}\bar{\omega}^m_{q-m}\sum_{\beta \in \mathcal{G}^m_{q-m}}\phi(\beta h) = a(h)\sum_{\gamma\in\mathcal{I}_p}\bar{\omega}_\gamma\sum_{\beta\in\mathcal{G}_\gamma}\phi(\beta h). 
\end{align} 
% Note that the left-hand side of equation (\cref{RawLinearSystem}) can be written as \begin{align*}
%     a(h)\sum_{\gamma\in\mathcal{I}_p}\omega_{\gamma}(h)\sum_{\beta\in\mathcal{G}_{\gamma}}(\beta h)^{2\xi}g(\beta h) = a(h)h^{2\xi}\sum_{\gamma\in\mathcal{I}_p}\omega_{\gamma}(h)\sum_{\beta\in\mathcal{G}_{\gamma}}(\beta)^{2\xi}g(\beta h),
% \end{align*} for each $\xi \in \mathcal{I}_p$.
% \subsection{On-diagonal Matrix Formulation}
To formulate \cref{RawLinearSystem} in matrix form, we re-index the set $\mathcal{I}_p=(\gamma_i)_{1\leq i\leq N_p}$. There are more than one way to index the set $\mathcal{I}_p$ and the indexing plays an important role in proving the linear system has unique solution. We specify the indexing later. By one-to-one correspondence between $\mathcal{G}$ and $\mathcal{I}_p$, we index $\mathcal{G}$  by $(\mathcal{G}_j\coloneqq \mathcal{G}_{\gamma_j})_{1\leq j\leq N_p}$, for each $\gamma_j\in\mathcal{I}_p, 1\leq j\leq N_p$. We denote $\boldsymbol{K},\boldsymbol{G}\in\mathbb{R}^{N_p\times N_p}$ such that for all $1\leq i,j\leq N_p$ \begin{align}\label{MatrixK}
    K_{i,j} = \sum_{\beta\in\mathcal{G}_j}\beta^{2\xi_i},\ G_{i,j} = g(\beta h)\delta_{i,j},\ \beta\in\mathcal{G}_j.
\end{align} We call $\boldsymbol{K}$ the coefficient matrix. We note that by radial symmetry, $g(\beta h)$ is constant over each $\mathcal{G}_j$, thus the choice of $\beta$ within $\mathcal{G}_j$ is arbitrary. Now, let \begin{align}
     \omega(h) = (\omega_1(h),\dots,\omega_{N_p}(h))^T,
\end{align} the linear system \cref{RawLinearSystem} becomes \begin{align}\label{MatrixLinearSystem}
    \boldsymbol{K}\boldsymbol{G}(h)\omega(h) = c(h),
\end{align} where the right-hand side of \cref{MatrixLinearSystem} \begin{align}
    c_i(h) &= \frac{1}{h^{2|\xi_i|_1}a(h)}\left(\int g(x)s(x)x^{2\xi_i}\d x - T_h^0[g\cdot s\cdot x^{2\xi_i}]\right),\quad 1\leq i\leq N_p, \label{RHSofLinearSystem} \\
    c(h) &= (c_1(h),\dots,c_{N_p}(h))^T.
\end{align} We solve the linear system \cref{MatrixLinearSystem} for each $h>0$ and the coefficients $\bar{\omega}^m_{q-m}:(q-m,m)\in\mathcal{I}_p$ are the limits of the solution of the linear system \cref{MatrixLinearSystem}: $\bar{\omega}=\lim_{h\to 0}\omega(h)$, provided the following claims hold \begin{itemize}
    \item the limit of right-hand side of equation \cref{RHSofLinearSystem} exists,
    \item the limit of $\boldsymbol{G}(h)$ exists, denoted by $\boldsymbol{G}$,
    \item matrices $\boldsymbol{K}$ and $\boldsymbol{G}$ are non-singular.
\end{itemize} The first two claims are answered in next section while the non-singularity of $\boldsymbol{K}$ is addressed in section \cref{sectionNSCM}.
% Given $p\in\mathbb{N}_0$, we re-write (\cref{FirstCorrectedTrapezRule}) into \begin{align}
%     Q_h^p[f] = T_h^0[f] + A_h^p[\phi]
% \end{align}where \begin{align} \label{A^p}
%     A^p_h[\phi]  = a(h)\sum_{q=0}^p\sum_{m=0}^{q}\bar{\omega}^m_{q-m}\sum_{\beta \in \mathcal{G}^m_{q-m}}\phi(\beta h) = a(h)\sum_{i=1}^{N_p}\bar{\omega}_i\sum_{\beta\in\mathcal{G}_i}\phi(\beta h). 
% \end{align} 

\subsection{Off-diagonal Corrected Trapezoidal Rule}
In this subsection, we construct corrected trapezoidal rule for the off-diagonal singular kernel. In $n$ dimensions, the off-diagonal singular kernel is defined by \begin{align}\label{eq:off singular kernel}
    s_{12}(x) = \frac{x_1x_2}{|x|^{n+\alpha}},\ 0<\alpha<2.
\end{align} In the absence of ambiguity, we recycle the notations used for the on-diagonal case. Due to similarity with on-diagonal case, we will simplify the description for off-diagonal case.

In two dimensions, we assign a correction weight $\omega_\beta$ on each grid points $\beta\coloneqq(\beta_1,\beta_2)\in\mathbb{Z}^2$ such that $\beta_1\beta_2\not = 0$. Imposing symmetry of $s_{12}$ on correction weights, we have \begin{align}
    \bar{\omega}_{\beta_1,\beta_2}  = \text{sgn}\{\beta_1\beta_2\}\bar{\omega}_{|\beta_1|,|\beta_2|},\quad \bar{\omega}_{\beta_1,\beta_2}  = \bar{\omega}_{\beta_2,\beta_1}.
\end{align} \Cref{fig:Off_diag_grid} illustrates the grid of the correction weights for $p = 6$ after imposing symmetry. We say that the grid of the off-diagonal correction weights for $p = N, N\geq 2$ is made up of $N-1$ correction layers, each layer is the set of points $\beta\in \mathbb{N}^2$ such that $|\beta|_1 = k,\ k = 2,\cdots,N$ and $\beta_1\geq\beta_2$.
\begin{figure}[h]
    \centering
    \includegraphics[width = 10cm]{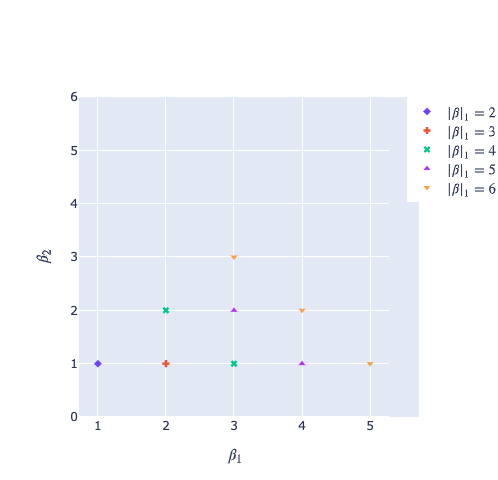}
    \caption{The location of the off-diagonal grid of correction weights after imposing symmetry when $p = 6$.}
    \label{fig:Off_diag_grid}
\end{figure}

For all $p\in\mathbb{N}, p\geq 2$, we denote \begin{align}
    \mathcal{G}_{q-m,m}  = & \big\{ (\pm(q-m),\pm m),(\mp(q-m),\pm m), \nonumber \\
    & (\pm m,\pm(q-m)),(\mp m,\pm(q-m))\big\},\ 2\leq q\leq p,\ 1\leq m \leq \left\lfloor\frac{q}{2}\right\rfloor, \label{eq:G_off} \\
    \mathcal{I}_p  = & \left\{(q-m,m):2\leq q\leq p, 1\leq m\leq \left\lfloor\frac{q}{2}\right\rfloor \right\}. \label{I_p off}
    \end{align} We define the modified trapezoidal rule for a regular function $\phi$ in $\mathbb{R}^2$ as \begin{align}
    Q_h^1[\phi\cdot s_{12}] & \coloneqq T_h^0[\phi\cdot s_{12}], \nonumber \\
    Q_{h}^p[\phi\cdot s_{12}] &\coloneqq T_h^0[\phi\cdot s_{12}] + h^{2-\alpha}\sum_{(q-m,m)\in\mathcal{I}_p}\bar{\omega}_{q-m,m} \sum_{\beta\in\mathcal{G}_{q-m,m}}\text{sgn}(\beta_1\beta_2)\phi(\beta h), \nonumber \\
    & \eqqcolon T_h^0[\phi\cdot s_{12}] + A_h^p[\phi],\ \text{if } p \geq 2. \label{CorrectedTrapz_off_diag}
\end{align} where $T_h^0[f]$ means punctured-hole trapezoidal rule applied on function $f$ defined at \cref{Punctured_Hole_Trapez}. To determine the correction weights, we introduce the linear system: for each $\xi\in\mathcal{I}_p$ \begin{multline}\label{LinearSysOffDiagonal_firstform}
     h^{2-\alpha}\sum_{(q-m,m)\in\mathcal{I}_p}\omega_{q-m,m}(h)\sum_{\beta\in\mathcal{G}_{q-m,m}}\text{sgn}(\beta_1\beta_2)(\beta h)^{2\xi-e_1-e_2}g(\beta h)  \\
    = \int_{\mathbb{R}^2\setminus\{0\}}g(x)s_{12}(x)x^{2\xi-e_1-e_2}\ \d x - T_h^0[g\cdot s_{12}\cdot x^{2\xi-e_1-e_2}],
\end{multline} where $g$ is a smooth, radially symmetric, compactly supported function on $\mathbb{R}^2$. We index $\mathcal{I}_p = (\gamma_j)_{1\leq j\leq N_p}, N_p:=|\mathcal{I}_p|$, which will be specified later. Then by using the same indexing we can write $\{\mathcal{G}_{q-m,m}:2\leq q\leq p, 1\leq m\leq \lfloor\frac{q}{2}\rfloor\} = \{\mathcal{G}_{\gamma_j}\}_j$. In this case the matrix $\boldsymbol{K}\in\mathbb{R}^{N_p\times N_p}$ is re-defined as\begin{align}
    K_{i,j} = \sum_{\beta\in\mathcal{G}_{\gamma_j}}\text{sgn}(\beta_1\beta_2)\beta^{2\xi_i-e_1-e_2}, \label{MatrixK_12}
\end{align} For the off-diagonal kernel $s_{12}$, the linear system \cref{MatrixLinearSystem} has a new $\boldsymbol{K}$ and $c(h)$ \begin{align}
    \boldsymbol{K}\boldsymbol{G}(h)\omega(h) &= c(h), \label{LinearSysOff_Second} \\
    c_j(h) &= \frac{1}{h^{2|\xi_j|_1-\alpha}}\left(\int g\cdot s_{12}\cdot x^{2\xi_j-e_1-e_2}\ \d x - T_h^0[g\cdot s_{12}\cdot x^{2\xi_j-e_1-e_2}]\right).\label{eq:c_h off}
\end{align} In \cref{sectionTA,,sectionNSCM} we will address the existence of the limit $\lim_{h\to 0}c(h)$, the convergence of $\lim_{h\to 0}\omega(h)$ and non-singularity of coefficient matrix, respectively.

%%%%%%%%%%%%%%%%%%%%%%%%%%%%%%%%%%%%%%%%%%%%%%%%%%%%%%%%%%%%%%%%%%%%%%%%%%%%%%%%%%%%%%%%%%%%%%%%%%%%%%%%%%%%%%%%%%%%%%%%%%%%%%%%%%%%%%%%%%%%%%5
\section{The Analysis of Orders of Convergence} \label{sectionTA}
\subsection{On-diagonal Analysis}
In this subsection, we show the convergence order of corrected trapezoidal rule $Q^p_h$ for on-diagonal singular integral in \cref{SecondCorrectedTrapezRule} is $2p+4-\alpha$ via the following main theorem:
\begin{theorem}\label{MainTheorem}
Let $p\in\mathbb{N}_0$. Given $\phi\in C_c^{N}(\mathbb{R}^2)$ such that $N\geq 2p+4$ and the on-diagonal singular kernel $s(x) = \frac{x_1^2}{|x|^{2+\alpha}},\ x\in\mathbb{R}^2$, $0<\alpha<2$. Denote $f \coloneqq \phi \cdot s$, let $g\in C_c^{\infty}(\mathbb{R}^2)$ such that $g$ is radially symmetric, $g(0)=1$ and $\partial^{k} g(0)=0$ for all $|k|_1 =1,\dots, 2p+1$. Let $\omega(h)$ be the solution of \cref{MatrixLinearSystem} for this $g$, then $\omega(h)$ converges to some $\bar{\omega}\in\mathbb{R}^{N_p}$ such that \begin{align}  
    |\omega(h)-\bar{\omega}|_2= O(h^{2p+2}).
\end{align} Moreover, there exists $C>0$ such that \begin{align} \label{AccuracyofCorrectedTrapezRule}
    \left|Q_h^p[f] - \int_{\mathbb{R}^2}f(x)\ \d x\right|\leq C h^{2p+4-\alpha}.
\end{align}
\end{theorem} 
We will introduce some definitions, lemmas and a theorem before we prove the main theorem \cref{MainTheorem}. 
\begin{definition}
Given $f\in L^1(\mathbb{R}^n)$, we define \begin{align}
    \widehat{f} (\xi) = \int_{\mathbb{R}^n}f(x)\exp{(-2\pi i \xi\cdot x)}\d x.
\end{align} $\widehat{f}$ is the Fourier transform of $f$.
\end{definition} We state the Poisson Summation formula for later use, for a proof, see \cite{grafakos2014classical}.
\begin{lemma}[Poisson Summation Formula]\label{PoissonSummationFormula}
Let $f$ be a continuous function on $\mathbb{R}^n$ which satisfies \begin{align}
    |f(x)|\leq \frac{C}{(1+|x|)^{n+\delta}} \quad \text{for some}\ C,\delta>0 \ \text{and for all}\ x\in\mathbb{R}^n,
\end{align} and whose Fourier transform $\widehat{f}$ restricted on $\mathbb{Z}^n$ satisfies \begin{align} \label{CondPoissonSummation}
    \sum_{m\in\mathbb{Z}^n}|\widehat{f}(m)|<\infty,
\end{align} then \begin{align} \label{PoissonSummationFormula3}
    \sum_{m\in\mathbb{Z}^n}\widehat{f}(m) = \sum_{k\in\mathbb{Z}^n}f(k).
\end{align}
\end{lemma}

It is not hard to verify that if $f\in C_c^{n+1}(\mathbb{R}^n)$, then $f$ satisfies the hypotheses of Poisson Summation formula, in fact, 
\begin{align}
     |f(x)| \leq \left(\frac{1+|R|}{1+|x|}\right)^{n+1}\|f\|_{\infty}, \ \text{where}\ supp(f)\subset B_R(0)\ \text{for some}\ R>0, 
\end{align}
and for fixed $m\in \mathbb{Z}^{n}\setminus\{0\}$, choose $|m_j| = \max_{i = 1,\cdots, n}|m_i|$, then by using integration by part with respect to $x_j$ $(n+1)$-times and the property of being compactly supported, we obtain
\begin{align}
    \widehat{f}(m) = (-1)^{n+1}\int_{\mathbb{R}^n}(\partial^{n+1}_{j}f)(x)\frac{\exp\{-2\pi i x \cdot m\}}{(-2\pi i m_j)^{n+1}}\ \d x,
\end{align} note that $|m| \leq \sqrt{n}|m_j|$, we immediately have 
\begin{align}
    |\widehat{f}(m)|\leq \left(\frac{\sqrt{n}}{2\pi}\right)^{n+1}\frac{\max_{|l|_1 = n+1}\|\partial^l f\|_{\mathcal{L}_1}}{|m|^{n+1}}.
\end{align} The right hand side of above inequality is summable over $\mathbb{Z}^n\setminus\{0\}$. Thus the condition \cref{CondPoissonSummation} is satisfied.

Now, we let $f \coloneqq g\cdot s$ where $g\in C_c^\infty(\mathbb{R}^n)$ and $s(x)$ is defined in \cref{SingularFunctionGeneral}. Introducing a radially symmetric cut-off function $\Psi\in C^\infty(\mathbb{R}^n)$ such that \begin{align}\label{Cut-offFunction}
    \Psi(x) = \left\{ \begin{array}{cc}
        0, & |x|\leq \frac{1}{2} \\
        1, & |x|\geq 1
    \end{array}\right. .
\end{align} A concrete example of such cut-off function is \begin{align}
    \Psi(x) = \left\{ \begin{array}{cc}
        0, & |x|\leq \frac{1}{2} \\
        \exp(\frac{-36}{9 - 16(1 - |x|^2)^2}+4), & \frac{1}{2}\leq |x|\leq 1 \\
        1, & |x|\geq 1
    \end{array}\right. .
\end{align} By the property of $\Psi$, $s(x)\Psi\left(\frac{x}{h}\right)$ can be continuously extend to the origin by letting $s(0)\Psi(0) = \lim_{x\to 0}s(x)\Psi\left(\frac{x}{h}\right)=0$. Therefore, \begin{align}
    T_h^0[f] & = h^n\sum_{\beta\in\mathbb{Z}^n\setminus\{0\}} f(\beta h)  \nonumber \\
    & = h^n\sum_{\beta\in\mathbb{Z}^n\setminus\{0\}} f(\beta h)\Psi(\beta) + h^n \underbrace{f(0)\Psi(0)}_{=0} = T_h\left[f(\cdot)\Psi\left(\frac{\cdot}{h}\right) \right].
\end{align} Hence we can split \begin{align} \label{SplitbyPsi}
    \int f(x)\d x - T_h^0[f] & =  \int f(x)\left(1-\Psi\left(\frac{x}{h}\right)\right)\d x \nonumber \\
    & + \int f(x)\Psi\left(\frac{x}{h}\right)\d x - T_h\left[f(\cdot) \Psi\left(\frac{\cdot}{h}\right)\right].
\end{align} We now state and prove \cref{SecondMainThm}, which establishes the existence of the limit $\lim_{h\to 0}c(h)$ from \cref{RHSofLinearSystem} and plays an important role in the proof of \cref{MainTheorem}. Before doing so, we need the following lemma.
\begin{lemma} \label{FirstLemmainErrorBound}
Let $\eta\in \mathbb{N}_0^n$, the partial derivatives of order $k$ $(\in\mathbb{N}_0^n)$ of $x\rightarrow \frac{x^{\eta} x_1^2}{|x|^{n+\alpha}}$, $x \in \mathbb{R}^n$ are given as \begin{align}\label{ErrorBoundLemma1}
    \partial^{k}\frac{x^{\eta} x_1^2}{|x|^{n+\alpha}} = \frac{1}{|x|^{n-2+\alpha - |\eta|_1+|k|_1}}P_{k,\eta}\left(\frac{x}{|x|}\right),
\end{align} where $P_{k,\eta}$ is a polynomial in $n$-variables with $\text{deg}P_{k,\eta} = 2 + |\eta|_1 + |k|_1$.
\end{lemma}
\begin{proof}
For $k=0$, \begin{align}\label{Lemma2.4Case0}
    \partial^{0}\frac{x^{\eta} x_1^2}{|x|^{n+\alpha}} = \frac{1}{|x|^{n-2+\alpha - |\eta|_1}}\left(\frac{x}{|x|}\right)^{\eta + 2e_1},
\end{align} which shows \cref{Lemma2.4Case0} is true for $k=0$ with $P_{0,\eta}(x) = x^{\eta+2e_1}$. Suppose for $k\in\mathbb{N}_0^n$ the equation \cref{ErrorBoundLemma1} is satisfied. We define $P_{k,\eta}^{(j)}\coloneqq \partial_{x_j}P_{k,\eta}$ with $j=1,\cdots,n$. Note that $P_{k,\eta}^{(j)}$ is a polynomial with one degree less than $P_{k,\eta}$. Now, \begin{align}
    \partial^{k+e_i}\left\{ \frac{x^{\eta+2e_1}}{|x|^{n+\alpha}}\right\} & = \partial_i\left\{ \frac{1}{|x|^{n-2+\alpha - |\eta|_1+|k|_1}}P_{k,\eta}\left(\frac{x}{|x|}\right)\right\} \nonumber \\
    & = -(n-2+\alpha-|\eta|_1+|k|_1)\frac{1}{|x|^{n-1+\alpha-|\eta|_1+|k|_1}}\frac{x_i}{|x|}P_{k,\eta}\left(\frac{x}{|x|}\right) \nonumber \\
    & + \frac{1}{|x|^{n-1+\alpha - |\eta|_1+|k|_1}}\left[ P_{k,\eta}^{(i)}\left(\frac{x}{|x|}\right) \left(1-\frac{x_i^2}{|x|^2}\right) - \sum_{j\not = i}P_{k,\eta}^{(j)}\left(\frac{x}{|x|}\right)\frac{x_i x_j}{|x|^2} \right] \nonumber \\
    & = \frac{1}{|x|^{n-1+\alpha-|\eta|_1+|k|_1}}\Bigg[ -(n-2+\alpha-|\eta|_1+|k|_1)\frac{x_i}{|x|}P_{k,\eta}\left(\frac{x}{|x|}\right) \nonumber \\
    & + \left(1-\frac{x_i^2}{|x|^2}\right)P_{k,\eta}^{(i)}\left(\frac{x}{|x|}\right) - \sum_{j\not = i}\frac{x_i x_j}{|x|^2}P^{(j)}_{k,\eta}\left( \frac{x}{|x|}\right) \Bigg] \nonumber \\
    & \eqqcolon \frac{1}{|x|^{n-2+\alpha-|\eta|_1+|k|_1+1}} P_{k+e_i,\eta}\left(\frac{x}{|x|}\right).
\end{align} One can check that $\text{deg}(P_{k+e_i,\eta})=3+|\eta|_1+|k|_1$. The proof is then completed by induction.
\end{proof}
In this work we follow the convention that the generic constant $C$ may vary from line to line in the proof.
\begin{theorem} \label{SecondMainThm}
Let $n\in\mathbb{N}$, $p\in\mathbb{N}_0$ and $M =\max(4p+4,n+1)$. Assume $g\in C_c^{N}(\mathbb{R}^n):N\geq M$ such that $\partial^{k}g(0) = 0$ for all $|k|_1=1,\dots,2p+1$. For any $\xi\in\mathbb{N}_0^n$ with $|\xi|_1\leq p$, we have \begin{align}
    \left| \int_{\mathbb{R}^n} g\cdot s \cdot x^{2\xi}\ \d x - T_h^0[g\cdot s\cdot x^{2\xi}] - g(0)h^{2|\xi|_1+2-\alpha}c(2\xi,\alpha)\right|\leq C h^{2p+4-\alpha + 2|\xi|_1},
\end{align} where $c(\xi,\alpha)\in\mathbb{R}$ and the constant $C$ depends only on $g$ and $p$.
\end{theorem}
% \begin{theorem} \label{SecondMainThm}
% Assume $g\in C_c^{N}(\mathbb{R}^n)$ such that $\partial^{k}g(0) = 0$ for all $|k|_1=1,2,\dots,2p+1$, where $N,p\in\mathbb{N}_0:2p\geq n-2\ \text{and}\ N\geq \max\{2p+4-\alpha, 2n\}$. Fix $\xi\in\mathbb{N}_0^n:|\xi|_1\leq p$, then \begin{align}
%     \left| \int_{\mathbb{R}^n} g\cdot s \cdot x^{2\xi}\ \d x - T_h^0[g\cdot s\cdot x^{2\xi}] - g(0)h^{2|\xi|_1+2-\alpha}c(\xi,\alpha)\right|\leq C h^{2p+4-\alpha},
% \end{align} where $c(\xi,\alpha)\in\mathbb{R}$ is a constant.
% \end{theorem} 
\begin{proof}
    Let $f = g\cdot s_{\xi}$ where $s_{\xi}(x) = s(x)x^{2\xi}$. By the property of $\Psi$ in \cref{Cut-offFunction}, the first term in \cref{SplitbyPsi} can be computed by \begin{align}\label{Splitterm1}
    & \int f(x)\left(1-\Psi\left(\frac{x}{h}\right)\right) \d x = \int_{B(0,h)} g(x)s_{\xi}(x)\left(1-\Psi\left(\frac{x}{h}\right)\right) \d x \nonumber \\
    & = h^{n}\int_{|x|\leq 1} g(hx)s_{\xi}(hx)\left(1-\Psi(x)\right) \d x 
    = h^{2|\xi|_1+2-\alpha}\int_{|x|\leq 1} g(hx)s_{\xi}(x)\left(1-\Psi(x)\right) \d x.
\end{align} From the assumptions on $g$ and $2p+2|\xi|_1+3-\alpha > 0$, Taylor's theorem and $n$-dimensional spherical co-ordinate transform, with radial direction denoted by $r$, we have \begin{align} \label{Splitterm2}
    &\left|\int_{|x|\leq 1}(g(hx)-g(0))s_{\xi}(x)(1-\Psi(x))\d x \right| \nonumber \\
    &= \left|\int_{|x|\leq 1}\left(g(hx)-\sum_{\substack{k\in\mathbb{N}_0^n,\\|k|_1\leq 2p+1}}\frac{\partial^{k}g(0)}{k!\partial x^{k}}(hx)^{k}\right)s_{\xi}(x)(1-\Psi(x))\d x\right|  \nonumber\\
    &\leq h^{2p+2} \sum_{\substack{k\in\mathbb{N}_0^n,\\|k|_1= 2p+2}}\int_{|x|\leq 1}\left|\frac{\partial^{k}g(\rho)}{k!\partial x^{k}}\right||x|^{2p+2|\xi|_1-n+4-\alpha}|1-\Psi(x)| \d x  \nonumber \\
    &\leq C \max_{|k|_1=2p+2}\|\partial^{k}g\|_{\infty}\ h^{2p+2}\int_0^1 r^{2p+2|\xi|_1+3-\alpha}\ \d r \leq  C\ h^{2p+2},
\end{align} where $\rho$ is some point between $hx$ and $0$ . Combining \cref{Splitterm1} and \cref{Splitterm2}, we have \begin{align} \label{FirstPartEstimate}
    \left|\int f(x)\left(1-\Psi\left(\frac{x}{h}\right)\right)\d x - g(0)h^{2|\xi|_1+2-\alpha}\int_{|x|\leq 1}s_{\xi}(x)(1-\Psi(x))\d x\right| 
    \leq C h^{2p+2|\xi|_1+4-\alpha}.
    \end{align} We define a dilation operator $(\tau^a\theta)(x)\coloneqq \theta(ax)$ where $a>0$ and $\theta$ is any measurable function on $\mathbb{R}^n$. Noting that $f_{\Psi,h}\coloneqq f(\cdot)\Psi(\frac{\cdot}{h})\in C_c^{N}(\mathbb{R}^n)$, then so is $\tau^hf_{\Psi,h}$. Hence, we apply Poisson Summation formula \cref{PoissonSummationFormula} to $\tau^hf_{\Psi,h}$ and get
    % hence there exists $R>0$ such that $\text{supp}(\tau^hf_{\Psi,h})\subset B(0,R)$ and \begin{align}
    %     |(\tau^hf_{\Psi,h})(x)|\leq \left(\frac{1+|R|}{1+|x|}\right)^{n+1}\|\tau^hf_{\Psi,h}\|_{\infty}.
    % \end{align} Suppose \begin{align}\label{AbsSumRequirement}
    %     \sum_{m\in\mathbb{Z}^n}|\widehat{(\tau^hf_{\Psi,h})}(m)|<\infty,
    % \end{align} then applying Poisson summation formula Lemma \cref{PoissonSummationFormula}, 
    \begin{align} \label{PoissonSummationFormula2}
        \sum_{\beta\in\mathbb{Z}^n}(\tau^hf_{\Psi,h})(\beta) 
         = \sum_{k\in\mathbb{Z}^n}\widehat{\tau^hf_{\Psi,h}}(k),
    \end{align} thus \begin{align}\label{DecompositionOfT_h}
        T_h[f_{\Psi,h}] &= T_h\left[f(\cdot)\Psi(\frac{\cdot}{h})\right]  
         = h^n\sum_{\beta\in\mathbb{Z}^n}f(h\beta)\Psi(\beta)  
         = h^n\sum_{\beta\in\mathbb{Z}^n}(\tau^hf_{\Psi,h})(\beta)  \nonumber \\
        & = h^n \sum_{k\in\mathbb{Z}^n}\widehat{\tau^hf_{\Psi,h}}(k) 
         = h^n\sum_{k\in\mathbb{Z}^n} h^{-n}\widehat{f_{\Psi,h}}\left(\frac{k}{h}\right) \nonumber \\
        & = \int_{\mathbb{R}^n} f(x)\Psi\left(\frac{x}{h}\right)\d x + \sum_{k\in\mathbb{Z}^n\setminus\{0\}}\widehat{f_{\Psi,h}}\left(\frac{k}{h}\right).
    \end{align}
    % Taking a closer look at formula (\cref{DecompositionOfT_h[f_(Psi,h)]}) we found (\cref{AbsSumRequirement}) holds if and only if \begin{align} \label{AbsSumEquivCond}
    %     \sum_{k\in\mathbb{Z}^n\setminus\{0\}}\left|\widehat{f_{\Psi,h}}\left(\frac{k}{h}\right)\right|<\infty.
    % \end{align} We now prove the (\cref{AbsSumEquivCond}) holds by showing a particular re-arrangement series converges absolutely. 
    Denote $I_{\xi}(h,k)\coloneqq \widehat{f_{\Psi,h}}\left(\frac{k}{h}\right)$, from \cref{PoissonSummationFormula2} and \cref{DecompositionOfT_h} we know 
    \begin{align}
        \sum_{k\in\mathbb{Z}^n\setminus\{0\}} |I_{\xi}(h,k)| &< \infty, \\
        T_h[f_{\Psi,h}] - \int_{\mathbb{R}^n} f(x)\Psi\left(\frac{x}{h}\right) \ \d x & = \sum_{k\in\mathbb{Z}^n\setminus\{0\}} I_{\xi}(h,k). \label{TrapzIntegralDiff}
    \end{align}
    % The set $\mathbb{Z}^n\setminus\{0\}$ can be decomposed into $n$ disjoint subsets, i.e., $\mathbb{Z}^n\setminus\{0\}=\bigcup_{i=1}^{n}A_i$ where \begin{align}
    %     A_i= \{k\in\mathbb{Z}^n\setminus\{0\}: k_m=0 \ \text{for all}\ m\ \text{but}\ i\text{'s many}\}.
    % \end{align} For example, \begin{align*}
    %     A_1 & = \{k = (0,\dots,0,k_m,0,\dots,0)\in\mathbb{Z}^n:k_m\not =0\ \text{for some}\ m\ \text{with }1\leq m\leq n\}, \\
    %     A_n & = \{k = (k_1,\dots,k_n)\in\mathbb{Z}^n:k_m\not = 0\ \text{for all}\ m\text{'s with }1\leq m\leq n\}.
    % \end{align*} Clearly, each $A_i$ can be further decomposed into $\binom{n}{i}$ cases, i.e. $A_i = \bigcup_{j=1}^{\binom{n}{i}}A_{i,j}$ where \begin{multline}
    %     A_{i,j}\coloneqq \{k\in\mathbb{Z}^n\setminus\{0\}:  k_m=0 \ \forall  m\not\in\{m^{(j)}_1,\dots,m^{(j)}_i\}, \\
    %      k_m\not = 0 \ \forall  m\in\{m^{(j)}_1,\dots,m^{(j)}_i\},\ 1\leq \ m^{(j)}_r\leq n, \ 1\leq r\leq i\}.
    % \end{multline} Therefore, \begin{align}
    %     \sum_{k\in\mathbb{Z}^n\setminus\{0\}}I_{\xi}(h,k) = \sum_{i=1}^n\sum_{j=1}^{\binom{n}{i}}\sum_{k\in A_{i,j}}I_{\xi}(h,k).
    % \end{align} 
    The following lemma provides an error estimate for each $I_{\xi}(h,k)$. 
    \begin{lemma}\label{TechnicalLemma}
Assume all conditions stated in \Cref{SecondMainThm}. For each $k\in\mathbb{Z}^n\setminus\{0\}$, let $k_j=\max_{l=1,\cdots,n}|k_l|$.
There exists a function $W(\xi,k)$, independent of $h\ \text{and}\ g$, and constant $C$, depending on $g$ and $M$, such that \begin{align}
    \left|I_{\xi}(h,k) - g(0)W(\xi,k)h^{2|\xi|_1+2-\alpha}(2\pi k_j)^{-M}\right|\leq C |k|^{-M}(h^{2p+4-\alpha+2|\xi|_1}+h^{M}).
\end{align}
\end{lemma} \begin{proof}[Proof of \cref{TechnicalLemma}]
     Denote $\hat{p} = 2p+1$. By using $\exp(i\lambda z) = \frac{1}{(i\lambda)^M}\frac{\d\empty^M}{\d z^M}(\exp(i\lambda z)),\ \lambda,z\in\mathbb{R}\setminus\{0\}$ and integration by parts repeatedly, \begin{align}
        I_{\xi}(h,k) &= \int_{\mathbb{R}^n}\exp\left(-\frac{2\pi i}{h}k\cdot x\right) g(x)s_{\xi}(x)\Psi\left(\frac{x}{h}\right)\d x \nonumber \\
        &= \left(\frac{-ih}{2\pi k_j}\right)^{M}\int_{\mathbb{R}^n}\partial^M_j\left(\Psi\left(\frac{x}{h}\right)g(x)s_{\xi}(x)\right)\exp\left(\frac{-2\pi i}{h}k\cdot x\right)\d x.
    \end{align} Here, $i$ is the imaginary unit and $\partial_j$ means taking derivative with respect to $x_j$. Define \begin{align}
        W(\xi,k)\coloneqq (-i)^{M}\int_{\mathbb{R}^n}\partial^M_j\left(\Psi \cdot s_{\xi}\right)(x)\exp\left(-2\pi i k \cdot x\right)\d x.
    \end{align} In the rest of the proof we drop the $\xi$ in first argument of $W$ for convenience. We note that $W$ does not depend on $h$ or $g$ and $\partial_j^{l}\Psi\in C_c^{\infty}$ for any $l>0$. By the condition on $M$ we know $M-(2|\xi|_1 + 2-\alpha) > M-(2p+2) \geq 2p+2 > 1$. It follows from \cref{FirstLemmainErrorBound} that \begin{align}
        \int_{|x|>1}|\partial^{M}_j s_{\xi}(x)|\d x &= \int_{|x|>1}\frac{1}{|x|^{n-2+\alpha-2|\xi|_1 + M}}\left|P_{M,\xi}\left(\frac{x}{|x|}\right)\right|\d x \nonumber \\
        &\leq C\int_1^{\infty}r^{1-\alpha+2|\xi|_1-M}\d r <\infty.
    \end{align} Therefore, \begin{align}\label{EstimateW(k)}
        |W(k)| \leq C\left(\int_{\mathbb{R}^n}\frac{\Psi(x)}{|x|^{n-2+\alpha-2|\xi|_1+M}}\d x +\sum_{0<l\leq M}\int_{\mathbb{R}^n}\partial^{l}_j \Psi(x)\partial^{M-l}_j s_{\xi}(x)\d x \right)<\infty.
    \end{align} The estimate \cref{EstimateW(k)} shows the boundedness of $W$ is independent of $k$. Now we re-scale the integral \begin{multline}\label{RescaledW(k)}
        W(k)  = (-i)^{M} h^{M-n}\int_{\mathbb{R}^n}\partial^{M}_j\left(\Psi\left(\frac{x}{h}\right)s_{\xi}\left(\frac{x}{h}\right)\right)\exp\left(\frac{-2\pi i}{h}k \cdot x\right)\d x  \\
         = (-i)^{M} h^{M+\alpha-2-2|\xi|_1}\int_{\mathbb{R}^n}\partial^{M}_j\left(\Psi\left(\frac{x}{h}\right)s_{\xi}(x)\right)\exp\left(\frac{-2\pi i}{h}k\cdot x\right)\d x.
    \end{multline} Then, denote $r(x)\coloneqq \Psi(\sfrac{x}{h})(g(x)-g(0))$, from \cref{RescaledW(k)}
    \begin{align} \label{EstI_xi}
        &I_{\xi}(h,k) - g(0)W(k)h^{2|\xi|_1+2-\alpha}(2\pi k_j)^{-M}  \nonumber \\
        &= \left(\frac{-i h}{2\pi k_j}\right)^{M} \int_{\mathbb{R}^n}\partial^{M}_j \left\{r(x)s_{\xi}(x)\right\}\exp\left(\frac{-2\pi i}{h}k\cdot x\right)\ \d x. 
    \end{align} It can be seen that for all $0\leq m\leq M$, \begin{align}
        \partial^{m}_jr(x) = \left\{ \begin{array}{ccc}
            0 &,& 0\leq |x|\leq \frac{h}{2},\\
            \sum_{0\leq l\leq m}\binom{m}{l}\frac{1}{h^{m-l}}\partial^{m-l}_j\Psi\left(\frac{x}{h}\right)(\partial^{l}_jg(x)-\delta_{0,l}g(0)) &,& \frac{h}{2}\leq |x|\leq h \\
            \partial^{m}_jg(x)-\delta_{0,m}g(0) &,& |x|>h
        \end{array}\right. .
    \end{align} Here we used Kronecker delta $\delta_{k_1,k_2}=\mathbbm{1}_{k_1}(k_2)$. Let $\text{supp}(g)\subseteq B(0,R)$ for some $R>1$, then we have \begin{align}\label{partial^j[r(x) - g(0)Psi(x/h)]}
        & \int_{\mathbb{R}^n}\partial^{m}_jr(x)\partial_j^{M-m}s_{\xi}(x)\exp\left(\frac{-2\pi i}{h} k\cdot x\right)\d x \nonumber \\
        & = \sum_{0\leq l\leq m}\frac{\binom{m}{l}}{h^{m-l}} \int_{\frac{h}{2}\leq |x|\leq h}\partial^{m-l}_j \Psi\left(\frac{x}{h}\right)\Big[\partial^{l}_jg(x)-\delta_{0,l}g(0)\Big]\partial^{M-m}_j s_{\xi}(x)\exp\left(\frac{-2\pi i}{h} k\cdot x\right)\d x  \nonumber \\
        & + \int_{h<|x|\leq R}\Big[\partial^{m}_jg(x) - \delta_{0,m}g(0)\Big]\partial^{M-m}_j s_{\xi}(x)\exp\left(\frac{-2\pi i}{h} k\cdot x\right)\d x  \nonumber\\
        & - \delta_{0,m}g(0)\int_{|x|>R}\partial^{M-m}_j s_{\xi}(x) \exp\left(\frac{-2\pi i}{h} k\cdot x\right)\d x. 
    \end{align} Let $T_{m,1},T_{m,2},T_{m,3}$ be the first, second and third terms in RHS of \cref{partial^j[r(x) - g(0)Psi(x/h)]}. Since $(\partial^{l}_jg)(0)=0$ for all $l=1,\cdots,\hat{p}$, we have for all $l\leq \hat{p}$ \begin{align}
        |\partial^{l}_j g(x)-(\partial^{l}_jg)(0)|\leq C \max_{\beta\in\mathbb{N}_0^n:|\beta|_1=\hat{p}+1}\|\partial^{\beta}g\|_{\infty}|x|^{\hat{p}+1-l}.
    \end{align} It follows that for all $0\leq m\leq M$ \begin{align} \label{partial_deriv_g_estimate}
        |\partial^{m}_jg(x)-\delta_{0,m}(\partial^{m}_jg)(0)|\leq C |x|^{\max\{\hat{p}+1-m,0\}}.
    \end{align} For all $0\leq m\leq M$ and $0\leq l\leq m$, we estimate each term in $T_{m,1}$, from \cref{partial_deriv_g_estimate} and \cref{FirstLemmainErrorBound}, \begin{align}\label{FirstEstinExtraConvOrder}
        & \frac{1}{h^{m-l}}\left|\int_{\frac{h}{2}\leq |x|\leq h}\partial^{m-l}_j \Psi\left(\frac{x}{h}\right)\Big[\partial^{l}_j g(x)-\delta_{0,l}g(0)\Big]\partial^{M-m}_j s_{\xi}(x)\exp\left(\frac{-2\pi i}{h} k\cdot x\right)\d x\right|  \nonumber \\
        % & \leq \frac{C}{h^{m-l}}\int_{\frac{h}{2}\leq |x|\leq h} |x|^{-(n-2+\alpha-2|\xi|_1+M-m)+\max\{\hat{p}+1-l,0\}}\d x  \nonumber \\
        & \leq \frac{C}{h^{m-l}} h^n \int_{\frac{1}{2}\leq |x|\leq 1}|h x|^{-(n-2+\alpha-2|\xi|_1+M-m)+\max\{\hat{p}+1-l,0\}}\d x  \nonumber\\
        & \leq C h^{\hat{p}+3-\alpha+2|\xi|_1-M}\int_{\frac{1}{2}}^1 r^{\hat{p}+2-\alpha+2|\xi|_1+m+l-M}\d r  \nonumber \\
        & \leq C h^{\hat{p}+3-\alpha+2|\xi|_1-M}, 
    \end{align} where we used $h^{\max\{\gamma,0\}}\leq h^{\gamma}$ if $0<h\leq 1$. It follows that $|T_{m,1}|\leq C h^{\hat{p}+3-\alpha+2|\xi|_1-M}$.
    % For all $0<\alpha<2$ and multi-indices $\hat{p}+1< |j|_1\leq |q|_1$, it is clear that $h^{2-\alpha+|j|_1-|q|_1}\leq h^{3-\alpha+\hat{p}-|q|_1}$. 
    By a similar computation as in \cref{FirstEstinExtraConvOrder}, we have \begin{align}\label{SecondEstinExtraConvOrder}
        |T_{m,2}| & = \left|\int_{h<|x|\leq R}\Big[\partial^{m}_jg(x)-\delta_{0,m}g(0)\Big]\partial^{M-m}_j s_{\xi}(x)\exp\left(\frac{-2\pi i}{h} k\cdot x\right)\d x\right|  \nonumber \\
        & \leq C\int_{h<|x|\leq R} |x|^{\max\{\hat{p}+1-m,0\}-(n-2+\alpha-2|\xi|_1+M-m)}\d x \nonumber \\
        & \leq C\int_h^R r^{\max(\hat{p}+1-m, 0) + 1 - \alpha + 2|\xi|_1 + m - M}\d r  \nonumber \\
        & \leq C\int_{h}^1 r^{\hat{p}+2-\alpha+2|\xi|_1-M}\d r + C \nonumber \\
        % & \leq C_q\left\{ \begin{array}{cc}
        %   1,  &\text{if }\hat{p}+2-\alpha+2|\xi|_1-|q|_1 > -1, \\
        %     1 + |\log(h)|, &\text{if } \hat{p}+2-\alpha+2|\xi|_1-|q|_1 = -1, \\
        %     1 + h^{\hat{p}+3-\alpha+2|\xi|_1-|q|_1}, &\text{if } \hat{p}+2-\alpha+2|\xi|_1-|q|_1 < -1. \\
        % \end{array}\right.  \nonumber \\
        % & \leq C\left\{ \begin{array}{cc}
        %   \int_{h}^R r^{\hat{p}+2-\alpha-|q|_1+2|\xi|_1}\ \d r,  & |j|_1\leq \hat{p}+1 \\
        %     \int_{h}^R r^{1-\alpha+|j|_1-|q|_1+2|\xi|_1}\ \d r, &   |j|_1>\hat{p}+1\ \& \ \alpha\not = 1 \\
        %     \int_{h}^R r^{|j|_1-|q|_1+2|\xi|_1}\ \d r, &  |j|_1>\hat{p}+1\ \& \ |j|_1\not = |q|_1-2|\xi|_1-1 \ \& \ \alpha = 1 \\
        %     \int_{h}^R r^{-1}\ \d r, & |j|_1 = |q|_1-2|\xi|_1-1 \ \& \ \alpha = 1
        % \end{array}\right.  \nonumber \\
        % & \leq C_{q}\left\{ \begin{array}{cc}
        %     1 + h^{\hat{p}+3-\alpha-|q|_1+2|\xi|_1}, & |j|_1\leq \hat{p}+1\\
        %     1 + h^{2-\alpha+|j|_1-|q|_1+2|\xi|_1}, & |j|_1>\hat{p}+1\ \& \ \alpha\not = 1 \\
        %     1 + h^{|j|_1-|q|_1+1+2|\xi|_1}, & |j|_1>\hat{p}+1\ \& \ |j|_1\not = |q|_1-2|\xi|_1-1 \ \& \ \alpha = 1 \\
        %     1 + |\log(h)|, & |j|_1 = |q|_1-2|\xi|_1-1 \ \& \ \alpha = 1
        % \end{array}\right. \nonumber \\
        & \leq C (h^{\hat{p}+3-\alpha+2|\xi|_1-M}+1). 
    \end{align} 
    We used $M - (\hat{p} + 2 - \alpha + 2|\xi|_1) \geq 1 + \alpha$ in the last inequality.
    % It is easy to see that constants $C_q$ is in fact uniformly bounded for all $q$ by writing out the formula for $C_q$. 
    Lastly, we estimate \begin{align}\label{ThirdEstinExtraConvOrder}
        |T_{m,3}|= \left|\int_{|x|\geq R}\partial^{M}_j s_{\xi}(x)\exp\left(\frac{-2\pi i}{h} k\cdot x\right)\ \d x\right|\leq C\int_{|x|\geq R}\frac{\d x}{|x|^{n-2+\alpha-2|\xi|_1+M}} <\infty.
    \end{align} Putting together the results from \cref{FirstEstinExtraConvOrder} to \cref{ThirdEstinExtraConvOrder} and using $|k|\leq \sqrt{n}|k_j|$, we obtain from \cref{EstI_xi} \begin{align}
        &|I_{\xi}(h,k)-g(0)W(\xi,k)h^{2|\xi|_1+2-\alpha}(2\pi k_j)^{-M}| \nonumber \\
        &\leq C\left(\frac{h}{|k_j|}\right)^M\sum_{m=0}^M (|T_{m,1}| + |T_{m,2}| + |T_{m,3}|) \nonumber \\
        &\leq C \left(\frac{h}{|k_j|}\right)^{M}(h^{\hat{p}+3-\alpha+2|\xi|_1-M}+1) \nonumber \\
        &\leq C|k_j|^{-M}(h^{2p+4-\alpha+2|\xi|_1}+h^{M}) \leq C|k|^{-M}(h^{2p+4-\alpha+2|\xi|_1}+h^{M}). \end{align}
\end{proof} 

% For each $(i,j)$ with $1\leq i\leq n, 1\leq j \leq \binom{n}{i}$, we fix a multi-index $q_{i,j}\in \mathbb{N}_0^n$ satisfying \cref{TechnicalLemma}. 
% Such $q_{i,j}$ is always exists since $N \geq 4(p+1)n$, we can let $q_m = 4(p+1)$ for all $m\in \{m_1^{(j)},\cdots,m_i^{(j)}\}$, $q_m = 0$ for all $m\not \in \{m_1^{(j)},\cdots,m_i^{(j)}\}$ and $q_{i,j} = (q_m)_{m=1}^n$.
% i.e. 
% \begin{align*}
%     q_{i,j} \coloneqq \arg\min_{\substack{q\in\mathbb{N}_0^n,\\|q|_1 \geq 2p + 4 - \alpha, \\ 
%     q_m=0 \ \forall m\not\in\{m^{(j)}_1,\dots,m^{(j)}_i\}, \\ q_m \equiv 0\mod 2,\ q_m > 0 \ \forall m\in\{m^{(j)}_1,\dots,m^{(j)}_i\}.}} |q|_1\ .
% \end{align*}
% and choosing \begin{align}
%       2p+4-\alpha  \leq |q_{i,j}|_1 = N,
% \end{align} 
%noting that the choices are always possible, we can at least let \begin{align}
%     q_{i,j}(m) = 2(p+1) \quad \forall m\in\{m^{(j)}_1,\dots,m^{(j)}_i\}.
% \end{align} 
%     we have \begin{align} \label{SecondPartEstimate}
%     &\sum_{k\in A_{i,j}}\left|I_{\xi}(h,k)\right| \nonumber \\
%     & \leq \sum_{k\in A_{i,j}}\left|I_{\xi}(h,k) - g(0)h^{2|\xi|_1+2-\alpha}W_{i,j}(\xi,k)(2\pi k)^{-q_{i,j}}\right| 
%      + |g(0)|h^{2|\xi|_1+2-\alpha}\sum_{k\in A_{i,j}}\left|W_{i,j}(\xi,k)(2\pi k)^{-q_{i,j}}\right| \nonumber \\
%     & \leq C\sum_{k\in A_{i,j}}k^{-q_{i,j}}\leq C \left(\frac{\pi^2}{3}\right)^i. 
% \end{align} This proves \cref{AbsSumEquivCond} and hence \cref{AbsSumRequirement}. 
We define \begin{align}
    c(\xi,\alpha) &= c_1(\xi,\alpha) - c_2(\xi,\alpha), \label{WeightsAll}\\
    \text{where} \quad c_1(\xi,\alpha) &\coloneqq \int_{|x|\leq 1}s_{\xi}(x)(1-\Psi(x))\d x<\infty, \label{Weights1}\\
    c_2(\xi,\alpha) &\coloneqq \sum_{\substack{k\in\mathbb{Z}^n\setminus\{0\},\\ k_j = \max_l|k_l|}}W(\xi,k)(2\pi k_j)^{-M}<\infty. \label{Weights2}
\end{align} \Cref{Weights2} holds since for each $k\in\mathbb{Z}^n\setminus\{0\}$, $|k|\leq \sqrt{n}|k_j|$ and so \begin{align*}
    \sum_{k\not= 0}|k_j|^{-M}\lesssim \sum_{k\not = 0}|k|^{-M}\leq \sum_{k\not = 0}|k|^{-(n+1)}<\infty.
\end{align*} Using \cref{SplitbyPsi,FirstPartEstimate,DecompositionOfT_h,TrapzIntegralDiff}, \crefrange{WeightsAll}{Weights2} and \Cref{TechnicalLemma}, we obtain \begin{align}
    & \left|\int_{\mathbb{R}^n} f(x)\d x - T_h^0[f] - g(0)h^{2-\alpha+2|\xi|_1}c(\xi,\alpha)\right|  \nonumber \\
    & \leq \left|\int_{\mathbb{R}^n} f(x)\Big(1-\Psi\Big(\frac{x}{h}\Big)\Big)\ \d x - g(0)h^{2-\alpha+2|\xi|_1}c_1(\xi,\alpha)\right| \nonumber \\
    & + \left|\int_{\mathbb{R}^n} f(x)\Psi\Big(\frac{x}{h}\Big)\d x - T_h\Big[f(\cdot)\Psi\Big(\frac{\cdot}{h}\Big)\Big] + g(0)h^{2-\alpha+2|\xi|_1}c_2(\xi,\alpha)\right|  \nonumber \\
    & \overset{\cref{TrapzIntegralDiff}}{\leq} \left|\int_{\mathbb{R}^n} f(x)\Big(1-\Psi\Big(\frac{x}{h}\Big)\Big)d x - g(0)h^{2-\alpha+2|\xi|_1}c_1(\xi,\alpha)\right| 
     + \left|g(0)h^{2-\alpha+2|\xi|_1}c_2(\xi,\alpha) -\sum_{k\in\mathbb{Z}^n\setminus\{0\}}I_{\xi}(h,k)\right|  \nonumber \\
    & \overset{\cref{FirstPartEstimate},\cref{Weights2}}{\leq} Ch^{2p+4-\alpha+2|\xi|_1} + \sum_{\substack{k\in\mathbb{Z}^n,k\not = 0,\\ k_j = \max_l|k_l|}}|I_{\xi}(h,k) - g(0)h^{2-\alpha+2|\xi|_1}W(\xi,k)(2\pi k_j)^{-M}| \nonumber \\
    & \leq  Ch^{2p+4-\alpha+2|\xi|_1} + C(h^{2p+4-\alpha+2|\xi|_1}+h^{M})\sum_{\substack{k\in\mathbb{Z}^n,k\not = 0, \\k_j = \max_l|k_l|}}|k_j|^{-M} \nonumber \\
    & \leq Ch^{2p+4-\alpha+2|\xi|_1} + Ch^{2p+4-\alpha+2|\xi|_1}\sum_{k\in\mathbb{Z}^n,k\not = 0}|k|^{-(n+1)}
    \leq  Ch^{2p+4-\alpha + 2|\xi|_1}
\end{align} This concludes the proof.
\end{proof} 
Before we prove the \Cref{MainTheorem}, we need to know the order of convergence of the punctured-hole trapezoidal rule \cref{Punctured_Hole_Trapez}.
\begin{corollary}\label{corollary:estimate of trapez alone}
Let $p\in\mathbb{N}_0$. Assume $g\in C_c^N(\mathbb{R}^n)$ such that $N\geq \max\{2p+4,n+1\}$ and $(\partial^kg)(0) = 0$ for all $k\in\mathbb{N}_0^n:|k|_1\leq 2p+1$. Then \begin{align*}
    \left|\int_{\mathbb{R}^n}g\cdot s\d x - T_h^0[g\cdot s]\right|\leq Ch^{2p+4-\alpha}.
\end{align*}
\end{corollary}
\begin{proof}
    The proof is rather similar to that of \Cref{SecondMainThm}. We only present the essential parts. Let $\Psi$ be the cut-off function introduced in \cref{Cut-offFunction} and $\hat{p}=2p+1$. By \cref{SplitbyPsi,,TrapzIntegralDiff} for $\xi = 0$, we have \begin{align}
        \int_{\mathbb{R}^n}g(x)s(x)\d x - T_h^0[g\cdot s] &= \int_{\mathbb{R}^n}g(x)s(x)\left(1-\Psi\left(\frac{x}{h}\right)\right)\d x - \sum_{k\in\mathbb{Z}^n\setminus\{0\}}I_0(h,k), \label{Split term in corollary}\\ 
        \text{where }I_0(h,k) &= \int_{\mathbb{R}^n}\exp\left(\frac{-2\pi i}{h}k\cdot x\right)g(x)s(x)\Psi\left(\frac{x}{h}\right)\d x. \nonumber 
    \end{align} By using $g(0) = 0$ and property of $\Psi$, we have $|g(x)|\leq C|x|^{2p+2}$ and \begin{align}
        &\left|\int_{\mathbb{R}^n}g(x)s(x)\left(1-\Psi\left(\frac{x}{h}\right)\right)\d x\right| \leq C \int_{|x|\leq h} |x|^{2p+4-n-\alpha}\d x \nonumber 
        \\
        &\leq C\int_{0}^h r^{2p+3-\alpha}\d r\leq Ch^{2p+4-\alpha}. \label{Est Term 1 in Corollary}
    \end{align} We now consider $I_0(h,k)$. Fix $k\in\mathbb{Z}^n,k\not= 0$, let $k_j = \max_{i=1}^n|k_i|$. Repeated integration by parts with respect to $x_j$ yields \begin{align*}
        I_0(h,k) = \left(\frac{-i h}{2\pi} \right)^N\frac{1}{k_j^N}\int_{\mathbb{R}^n}\partial_j^N\left(\Psi\left(\frac{x}{h}\right)g(x)s(x)\right)\exp\left(\frac{-2\pi i}{h}k\cdot x\right)\d x.
    \end{align*} For a fixed $m$ with $0\leq m\leq N$, let $r(x) = \Psi(\sfrac{x}{h})g(x)$ and \begin{align*}
        I_m = \int_{\mathbb{R}^n}\partial_j^m r(x)\partial_j^{N-m}s(x)\exp\left(\frac{-2\pi i}{h}k\cdot x\right)\d x.
    \end{align*} Since \begin{align*}
        \partial^{m}_jr(x) = \left\{ \begin{array}{ccc}
            0 &,& 0\leq |x|\leq \frac{1}{2}\\
            \sum_{0\leq l\leq m}\binom{m}{l}\frac{1}{h^{m-l}}\partial^{m-l}_j\Psi\left(\frac{x}{h}\right)\partial^{l}_jg(x)&,& \frac{h}{2}\leq |x|\leq h \\
            \partial^{m}_jg(x)&,& |x|>h
        \end{array}\right. ,
    \end{align*} we have \begin{align*}
        |I_m| & \leq C \sum_{0\leq l\leq m}\frac{1}{h^{m-l}}\left|\int_{\frac{h}{2}\leq |x|\leq h}\partial^{m-l}_j \Psi\left(\frac{x}{h}\right)\partial^{l}_j g(x)\partial^{N-m}_j s(x)\exp\left(\frac{-2\pi i}{h} k\cdot x\right)\d x \right| \\
        & + \left|\int_{h<|x|<R}\partial^{m}_jg(x)\partial^{N-m}_j s(x)\exp\left(\frac{-2\pi i}{h} k\cdot x\right)\d x\right|,
    \end{align*} where $supp(g)\subset B(0,R)$. The argument of $|I_m|\leq C(h^{2p+4-\alpha-N} + 1)$ follows exactly \cref{FirstEstinExtraConvOrder,,SecondEstinExtraConvOrder} by replacing $M$ with $N$ and $\xi$ with $0$. Since $N\geq\max(2p+4,n+1)$, we obtain \begin{align}
        |I(h,k)|&\leq C\frac{1}{|k_j|^N}(h^{2p+4-\alpha} + h^N)\leq C \frac{1}{|k_j|^N}h^{2p+4-\alpha}, \nonumber \\
        \sum_{k\in\mathbb{Z}^n,k\not = 0}|I(h,k)|&\leq C h^{2p+4-\alpha}\sum_{\substack{k\in\mathbb{Z}^n,k\not = 0,\\ k_j=\max_l|k_l|}}\frac{1}{|k_j|^{n+1}}\leq Ch^{2p+4-\alpha}. \label{Est Term 2 in Corollary}
    \end{align} Combining \cref{Est Term 1 in Corollary,,Est Term 2 in Corollary}, by \cref{Split term in corollary} we have \begin{align}
        \left|\int_{\mathbb{R}^n}g\cdot s\d x - T_h^0[g\cdot s]\right|\leq Ch^{2p+4-\alpha}.
    \end{align}
\end{proof}

We are in the position to prove our \Cref{MainTheorem}.
% \begin{theorem}
% Let $p\in\mathbb{N}_0$. Given $\phi\in C_c^{N}(\mathbb{R}^2)$ such that $N\geq \max\{2p+4-\alpha, 4\}$, and the singular function $s(x) = \frac{x_1^2}{|x|^{2+\alpha}},\ x\in\mathbb{R}^2$, $0<\alpha<2$. Denote $f \coloneqq \phi \cdot s$ and let $g\in C_c^{N}(\mathbb{R}^2)$ such that $g$ is radially symmetric, $g(0)=1$ and $\partial^{k} g(0)=0$ for all $|k|_1 =1,\dots, 2p+1$. Let $\omega(h)$ be the solution of (\cref{MatrixLinearSystem}) for this $g$, then $\omega(h)$ converges to some $\bar{\omega}\in\mathbb{R}^{N_p}$ such that \begin{align}  
%     |\omega(h)-\bar{\omega}|_2= O(h^{2p+2}).
% \end{align} Moreover, there exists $C>0$ such that \begin{align}
%     \left|Q_h^p[f] - \int_{\mathbb{R}^2}f(x)\ \d x\right|\leq C h^{2p+4-\alpha}.
% \end{align}
% \end{theorem} 
% \begin{remark}
% The regularity requirement of $\phi$ is seemingly very strict. In fact, by scrutinizing the proof of Theorem \cref{SecondMainThm}, we found that when $\alpha \not = 1$, we only need $N \geq \max(2p + \lfloor 3 - \alpha\rfloor , 2n) $ where $n$ is underlying dimension.
% \end{remark}
\begin{proof}[Proof of  \cref{MainTheorem}]
We first show the limit of the solution $\omega(h)$ of the linear system \cref{MatrixLinearSystem} exists, i.e. $\lim_{h\to 0}\omega(h) = \bar{\omega}$. To this end, we apply \Cref{SecondMainThm} to each $c_i(h)$ in \cref{RHSofLinearSystem} and let $a(h) = h^{2-\alpha}$, yielding \begin{align}\label{Conv_of_c(h)}
        c_i(h) & = \frac{1}{h^{2|\xi_i|_1+2-\alpha}}\left(\int g(x)s(x)x^{2\xi_i}\ \d x - T_h^0[g\cdot s\cdot x^{2\xi_i}]\right)  \\
        & = \frac{1}{h^{2|\xi_i|_1+2-\alpha}} \left(g(0)h^{2|\xi_i|_1+2-\alpha}c(\xi_i,\alpha) + O(h^{2p+4-\alpha+2|\xi_i|_1})\right) \nonumber \\
        & = c(\xi_i,\alpha) + O(h^{2p+2})\underset{h\to 0}{\longrightarrow} c(\xi_i,\alpha). \nonumber 
    \end{align} Define \begin{align}\label{Const_C_on}
        \mathfrak{C}(\alpha) = (c(\xi_1,\alpha),\dots,c(\xi_{N_p},\alpha)).
    \end{align} It is also clear that \begin{align}
        \boldsymbol{G}(h) = \text{diag}(g(\xi_1h),\dots,g(\xi_{N_p}h)))\underset{h\to 0}{\longrightarrow} \boldsymbol{I}_{N_p\times N_p},
    \end{align} and so \begin{align}
        \boldsymbol{G}(h)^{-1} = \text{diag}(g(\xi_1h)^{-1},\dots,g(\xi_{N_p}h)^{-1}))\underset{h\to 0}{\longrightarrow} \boldsymbol{I}_{N_p\times N_p}.
    \end{align} In addition, $\|\boldsymbol{G}(h)^{-1}\|_2 = \max\{|g(\xi_kh)|^{-1}:1\leq k\leq N_p\}$ is bounded in a neighbourhood of origin. Define $\bar{\omega}$ to be the solution of the linear system \begin{align}\label{MatrixLinearSystem3}
        \boldsymbol{K}\bar{\omega} = \mathfrak{C}(\alpha).
    \end{align} 
    Assuming $\boldsymbol{K}$ is non-singular, which will be proved in next section, from \cref{MatrixLinearSystem} and \cref{MatrixLinearSystem3}, we obtain \begin{align}
        |\omega(h) - \bar{\omega}|_2 & = |\boldsymbol{G}(h)^{-1}\boldsymbol{K}^{-1}c(h) - \boldsymbol{K}^{-1}\mathfrak{C}(\alpha)|_2 \nonumber \\
        & = |\boldsymbol{G}(h)^{-1}(\boldsymbol{I}-\boldsymbol{G}(h))\boldsymbol{K}^{-1}c(h) + \boldsymbol{K}^{-1}(c(h)-\mathfrak{C}(\alpha))|_2 \nonumber \\
        & \leq \|\boldsymbol{G}(h)^{-1}\|_2 \ \|\boldsymbol{I}-\boldsymbol{G}(h)\|_2\ \|\boldsymbol{K}^{-1}\|_2 \ |c(h)|_2 + \|\boldsymbol{K}^{-1}\|_2\ |c(h)-\mathfrak{C}(\alpha)|_2. \label{Conv_Omega1}
    \end{align} A Taylor expansion of $g$ around the origin gives \begin{align}
        |1- g(\xi_k h)|\leq Ch^{2p+2}, \quad k=1,\dots,N_p, 
    \end{align} and hence \begin{align}
        \|\boldsymbol{I} - \boldsymbol{G}(h)\|_2 = \max\{|1-g(\xi_k h)|:k=1,\dots,N_p\} \leq Ch^{2p+2}. \label{Conv_I_minus_G(h)}
    \end{align} From \cref{Conv_of_c(h)}, \cref{Conv_Omega1} and \cref{Conv_I_minus_G(h)}, we therefore have \begin{align}\label{Conv_omega(h)}
        |\omega(h) - \bar{\omega}|_2\leq Ch^{2p+2},
    \end{align} where the constant $C$ depends only on $g$ and $p$.
    
    We now prove the accuracy of the corrected trapezoidal rule \cref{AccuracyofCorrectedTrapezRule}. We denote the Taylor polynomial of $\phi$ with degree $2p+1$ at the origin by \begin{align}
        P_{\phi}(x)=\sum_{\substack{\xi\in\mathbb{N}_0^2,\\ 0\leq |\xi|_1\leq 2p+1}}\frac{x^{\xi}}{\xi!}\frac{\partial^{\xi}\phi(0)}{\partial x^{\xi}},
    \end{align} and define $\tilde{\phi}=P_{\phi}\cdot g$. It is clear that \begin{align}\label{PhisTaylor}
    \partial^{k}[\phi(x)-\tilde{\phi}(x)]|_{x=0}=0,\quad \forall\ k\in\mathbb{N}_0^2:|k|_1=0,1,\dots,2p+1.
\end{align} Let $f\coloneqq \phi\cdot s$ and $I[f]\coloneqq \int_{\mathbb{R}^2} f(x)\d x$, we split $Q_h^p[f]-I[f]$ as \begin{align}
    Q_h^p[f]-I[f] & = A_h^p[\phi-\tilde{\phi}]+(T_h^0-I)[(\phi-\tilde{\phi})\cdot s]+ (Q_h^p-I)[\tilde{\phi}\cdot s] \nonumber \\
    & \eqqcolon E_1+E_2+E_3 .
\end{align} To estimate $E_1$, from \cref{PhisTaylor} and Taylor's theorem we have \begin{align*}
    |\phi(x)-\tilde{\phi}(x)|\leq \max_{\substack{x\in\mathbb{R}^2,\\ k\in\mathbb{N}_0^2,\\|k|_1=2p+2}}\left(\frac{|\partial^k(\phi - \tilde{\phi})(x)|}{k!}\right)\ |x|^{2p+2}.
\end{align*} Then from \cref{A^p} we get \begin{align}\label{E1Est}
    |E_1| = |A_h^p[\phi-\tilde{\phi}]| 
    \leq Ch^{2-\alpha}\max_{1\leq j\leq N_p}|\bar{\omega}_j|\max_{\xi\in\mathcal{G}_j:1\leq j\leq N_p}|\phi(\xi h)-\tilde{\phi}(\xi h)| 
    \leq Ch^{2p+4-\alpha}.
\end{align} In order to bound $E_2$, we apply  \Cref{corollary:estimate of trapez alone} to $\phi-\tilde{\phi}$, giving rise to \begin{align}\label{E2Est}
    |(T_h^0-I)[(\phi-\tilde{\phi})\cdot s]|\leq Ch^{2p+4-\alpha}.
\end{align} Finally, it remains to estimate $E_3$. Since \begin{align}\label{PreE_3}
    |E_3|&= |Q^p_h[\tilde{\phi}\cdot s] - I[\tilde{\phi}\cdot s]| \nonumber \\
    &\leq \sum_{\substack{\xi\in\mathbb{N}_0^2,\\0\leq|\xi|_1\leq 2p+1}}\frac{1}{\xi!}\left|\frac{\partial^{\xi}\phi(0)}{\partial x^{\xi}}\right||Q_h^p[g\cdot s\cdot x^{\xi}] - I[g\cdot s\cdot x^{\xi}]| \nonumber \\
    &\leq C\sum_{\substack{\xi\in\mathbb{N}_0^2,\\0\leq|\xi|_1\leq 2p+1}}|T_h^0[g\cdot s\cdot x^{\xi}] - I[g\cdot s\cdot x^{\xi}]+A^p_h[g\cdot x^{\xi}]|.
\end{align} We claim that $T_h^0[g\cdot s\cdot x^{\xi}]$, $I[g\cdot s\cdot x^{\xi}]$ and $A^p_h[g\cdot x^{\xi}]$ vanish when at least one of $\xi_1$ or $\xi_2$ is odd. To see this, by radial symmetry of $g$ and $\frac{1}{|x|^{2+\alpha}}$,  \begin{align} \label{OddpowerI}
    & I[g\cdot s\cdot x^{\xi}]  = \int_{\mathbb{R}^2}\frac{g(x)}{|x|^{2+\alpha}}x^{\xi+2e_1}\d x  \\
    & = \int_{[0,\infty)^2}\frac{g(x)}{|x|^{2+\alpha}}[x^{\xi+2e_1}+(-x_1)^{\xi_1+2}(x_2)^{\xi_2}+(x_1)^{\xi_1+2}(-x_2)^{\xi_2}+ (-x_1)^{\xi_1+2}(-x_2)^{\xi_2}] \d x\nonumber \\
    & = \int_{[0,\infty)^2}g(x)s(x)x^{\xi}[1+(-1)^{\xi_1}+(-1)^{\xi_2}+(-1)^{|\xi|_1}]\d x. \nonumber 
\end{align} It is easy to see that \cref{OddpowerI} vanish if at least one of $\xi_1$ or $\xi_2$ is odd. The argument for $T_h^0[g\cdot s\cdot x^{\xi}]=0$ under same condition for $\xi$ closely resembles above. Recall that $g(\xi h)$ is constant over each $\mathcal{G}_{q-m}^m$ and if $(a_1,a_2)\in\mathcal{G}_{q-m}^m$ then $(\pm a_1,\pm a_2),(\mp a_1,\pm a_2)\in\mathcal{G}_{q-m}^m$ for all $(q-m,m)\in\mathcal{I}_p$. Denote $g_{q-m,m} = g(\xi h),\ \forall \xi\in\mathcal{G}_{q-m}^m$. By \cref{A^p} we have \begin{align}
    & A^p_h[g\cdot x^{\xi}] = a(h)\sum_{(q-m,m)\in\mathcal{I}_p}\omega_{q-m}^m g_{q-m,m} \sum_{\beta\in\mathcal{G}_{q-m}^m}(\beta h)^{\xi}  \\
    &= a(h)\sum_{(q-m,m)\in\mathcal{I}_p}\omega_{q-m}^m g_{q-m,m} h^{|\xi|_1} \sum_{\beta\in\mathcal{G}_{q-m}^m}\beta^{\xi}\frac{1+(-1)^{\xi_1}+(-1)^{\xi_2}+(-1)^{|\xi|_1}}{4} \nonumber\\
    &= 0, \nonumber
\end{align} if at least one of $\xi_1$ or $\xi_2$ is odd. Therefore, we can re-write \cref{PreE_3} as \begin{align}\label{Pre2E3}
    |E_3| &\leq C\sum_{\substack{0\leq|\xi|_1\leq p,\\ \xi\in\mathbb{N}_0^2}}|T_h^0[g\cdot s\cdot x^{2\xi}] - I[g\cdot s\cdot x^{2\xi}]+A^p_h[g\cdot x^{2\xi}]|.
\end{align} By \cref{RawLinearSystem,,Conv_omega(h)} and $a(h) = h^{2-\alpha}$, $E_3$ can be estimated by \begin{align}\label{E3Est}
    |E_3|& \leq C a(h)\sum_{(q-m,m)\in\mathcal{I}_p}|\omega^m_{q-m}(h)-\bar{\omega}_{q-m}^m|\left(\sum_{\beta\in\mathcal{G}_{q-m}^m}|g(\beta h)||\beta h|^{2|\xi|_1}\right)  \\
    & \leq Ch^{2-\alpha}|\omega(h) - \bar{\omega}|_2 
    \leq Ch^{2p+4-\alpha}. \nonumber
\end{align}
Putting together \cref{E1Est,,E2Est,,E3Est}, we have proved \begin{align}
     |Q_h^p[f] - I[f]|\leq Ch^{2p+4-\alpha}.
\end{align}
\end{proof} 

The following corollary provides a Taylor's expansion of $c(h)$ as in \cref{RHSofLinearSystem}, it is used in Richardson extrapolation for numerical computation of correction weights. Recall that $\mathfrak{C}(\alpha) = \lim_{h\to 0}c(h)$ in \cref{Const_C_on}.
\begin{corollary}\label{Taylor_of_c(h)_on}
Let $p,k\in\mathbb{N}_0$. For any $g\in C_c^\infty(\mathbb{R}^2)$ such that $g(0) = 1$, $\partial^lg(0) = 0,\ l\in\mathbb{N}_0^2$ with $0\not = |l|_1\leq 2p+1$, we have\begin{align}
    c(h) = \mathfrak{C}(\alpha) + h^{2p+2}\sum_{j=0}^k C_jh^{2j},
\end{align} where $C_j$'s are constants independent of $h$.
\end{corollary}
\begin{proof}
    If $k=0$, this is proved in the  \Cref{MainTheorem}. Assume corollary holds for some $k\in\mathbb{N}$. We prove the corollary holds for $k+1$. Fix $\xi_i\in\mathcal{I}_p,\ i\in\{1,\dots,N_p\}$. It is suffices to show $c_i(h) = C(\xi_i,\alpha) + h^{2p+2}O(\sum_{j=0}^{k+1} h^{2j})$. To this end, let $g^*\in C_c^\infty(\mathbb{R}^2)$ such that $g^*(0) = 2$ and $\partial^lg^*(0) = 0,\ l\in\mathbb{N}_0^2:0\not = |l|_1\leq 2(p+k)+3$, then by  \Cref{SecondMainThm} we have \begin{align}\label{Formula1_small_corollary_on}
        \left|\frac{1}{h^{2|\xi_i|_1+2-\alpha}}\left[\int_{\mathbb{R}^2}g^*(x)s(x)x^{2\xi_i}\d x - T_h^0[g^*\cdot s\cdot x^{2\xi_i}] \right] - 2C(\xi_i,\alpha) \right|\lesssim h^{2p+2+2k+2}.
    \end{align} Let $\bar{g} = g^* - g$, then $\bar{g}(0) = 1,\ \partial^l \bar{g}(0) = 0$ for all $l\in\mathbb{N}^2_0:0\not = |l|_1\leq 2p+1$. Applying inductive hypothesis to $\bar{g}$ we have \begin{align}\label{Formula2_small_corollary_on}
        \left|\frac{1}{h^{2|\xi_i|_1+2-\alpha}}\left[\int_{\mathbb{R}^2}\bar{g}(x)s(x)x^{2\xi_i}\d x - T_h^0[\bar{g}\cdot s\cdot x^{2\xi_i}] \right]  - C(\xi_i,\alpha)\right|\leq h^{2p+2}O\left(\sum_{j=0}^k h^{2j}\right).
    \end{align} Summing \cref{Formula1_small_corollary_on} with \cref{Formula2_small_corollary_on}, 
    \begin{align}
        |c_i(h) - C(\xi_i,\alpha)| & =\left| \frac{1}{h^{2|\xi_i|_1+2-\alpha}}\left[\int_{\mathbb{R}^2}g(x)s(x)x^{2\xi_i}\d x - T_h^0[g\cdot s\cdot x^{2\xi_i}]\right] - C(\xi_i,\alpha)\right| \\
        & \leq h^{2p+2}O\left(\sum_{j=0}^{k+1} h^{2j}\right). \nonumber
    \end{align} This completes the induction.
\end{proof}

\subsection{Off-diagonal Analysis}
In this subsection we prove order of convergence of corrected trapezoidal rule for off-diagonal singular integral \cref{CorrectedTrapz_off_diag} is $2p+2-\alpha, \ p\geq 1$. Note that in off-diagonal case, it turns out that the correction weights for $\xi = (0,0), (0,1)$ are zero and hence the correction for punctured-hole trapezoidal rule is automatic at these grid points. As a consequence, the number of correction layers is $p-1$, compared with $p+1$ in the on-diagonal case. We omit the proofs of \Cref{FirstLemma_off,,OffdiagSecondThm,,corollary:order trapez off}, as they are almost identical with that of \Cref{FirstLemmainErrorBound,,SecondMainThm,,corollary:estimate of trapez alone}, respectively.
\begin{lemma} \label{FirstLemma_off}
For all $k,\eta\in\mathbb{N}_0^n$, we have \begin{align}
    \partial^{k}\left\{\frac{x^{\eta+e_1+e_2}}{|x|^{n+\alpha}}\right\} = \frac{1}{|x|^{n+\alpha-2-|\eta|_1+|k|_1}}P_{k,\eta}\left(\frac{x}{|x|}\right),
\end{align} where $P_{k,\eta}$ is a polynomial with $\deg{P_{k,\eta}} = 2+|k|_1+|\eta|_1$.
\end{lemma}

\begin{theorem}\label{OffdiagSecondThm}
Let $n\in\mathbb{N}$, $p\in\mathbb{N}$ and $M=\max(4p,n+1)$. Assume $g\in C_c^{N}(\mathbb{R}^n):N\geq M$ with $\partial^{k}g(0) = 0$ for all $|k|_1=1,2,\dots,2p-1$. For any $\xi\in\mathbb{N}_0^n:|\xi|_1\leq 2p-2$, we have \begin{align}
    \Big| \int_{\mathbb{R}^n} g\cdot s_{12} \cdot x^{\xi}\d x  
    - T_h^0[g\cdot s_{12}\cdot x^{\xi}]
    - g(0)h^{|\xi|_1+2-\alpha}c(\xi,\alpha)\Big|\leq C h^{2p+2-\alpha+|\xi|_1},
\end{align} where $c(\xi,\alpha)\in\mathbb{R}$ is a constant.
\end{theorem}
\begin{remark}
By replacing $M,\xi$ in \cref{SecondMainThm} with the ones given in \cref{OffdiagSecondThm} and replacing \cref{FirstLemmainErrorBound} with \cref{FirstLemma_off}, the proof of \cref{OffdiagSecondThm} is verbatim proof of \cref{SecondMainThm}.
\end{remark}

\begin{corollary}\label{corollary:order trapez off}
Let $n\in\mathbb{N}$, $p\in\mathbb{N}$. Assume $g\in C_c^{N}(\mathbb{R}^n)$ with $N\geq \max(2p+2,n+1)$ such that $\partial^{k}g(0) = 0$ for all $|k|_1=0,1,\dots,2p-1$.Then \begin{align}
    \Big| \int_{\mathbb{R}^n} g\cdot s_{12}\d x  - T_h^0[g\cdot s_{12}]\Big|\leq C h^{2p+2-\alpha}.
\end{align} 
\end{corollary}

We now state and prove the main theorem for the off-diagonal case.
\begin{theorem}\label{OffdiagThirdThm}
Let $p\in\mathbb{N}$. Given $\phi\in C_c^{N}(\mathbb{R}^2)$ with $N\geq 2p+2$ and the singular function $s_{12}(x) = \frac{x_1x_2}{|x|^{2+\alpha}},\ x\in\mathbb{R}^2$, $0<\alpha<2$. Denote $f \coloneqq \phi \cdot s_{12}$ and let $g\in C_c^{\infty}(\mathbb{R}^2)$ such that $g$ is radially symmetric, $g(0)=1$ and $\partial^{k} g(0)=0$ for all $|k|_1 =1,\dots, 2p-1$. Let $\omega(h)$ be the solution of \cref{LinearSysOff_Second} for this $g$, then $\omega(h)$ converges to some $\bar{\omega}\in\mathbb{R}^{N_p}$ and the convergence rate is given by \begin{align}  
    |\omega(h)-\bar{\omega}|_2= O(h^{2p}).
\end{align} Moreover, there exists $C>0$ such that \begin{align}
    \left|Q_h^p[f] - \int_{\mathbb{R}^2}f(x)\ dx\right|\leq C h^{2p+2-\alpha}.
\end{align}
\end{theorem}
\begin{proof}
    We first show the limit of the solution $\omega(h)$ of the linear system \cref{LinearSysOff_Second} exists. To this end, we apply  \cref{OffdiagSecondThm} to each $c_i(h)$ in \cref{eq:c_h off}, yielding \begin{align}\label{Conv_c(h)_off}
        c_i(h) & = \frac{1}{h^{2|\xi_i|_1-\alpha}}\left(\int g(x)s(x)x^{2\xi_i-e_1-e_2}\ dx - T_h^0[g\cdot s\cdot x^{2\xi_i-e_1-e_2}]\right) \nonumber \\
        & = \frac{1}{h^{2|\xi_i|_1-\alpha}} \left(g(0)h^{2|\xi_i|_1-\alpha}c(\xi_i,\alpha) + O(h^{2p-\alpha+2|\xi_i|_1})\right) \nonumber \\
        & = c(\xi_i,\alpha) + O(h^{2p}).
    \end{align} Note that we write $c(2\xi_i-e_1-e_2,\alpha)$ as $c(\xi_i,\alpha)$ for convenience. Define \begin{align} \label{Const_C_off}
        \mathfrak{C}(\alpha) & \coloneqq (c(\xi_1,\alpha),\dots,c(\xi_{N_p},\alpha)),
    \end{align} and let $\bar{\omega}$ be the solution of the linear system \begin{align}\label{omega_linear_sys_off}
        \boldsymbol{K}\bar{\omega} = \mathfrak{C}(\alpha).
    \end{align} Assuming $\boldsymbol{K}$ is non-singular and using the arguments from \crefrange{Const_C_on}{Conv_Omega1}, we obtain \begin{align}
        |\omega(h) - \bar{\omega}|_2 
        \leq \|\boldsymbol{G}(h)^{-1}\|_2 \ \|\boldsymbol{I}-\boldsymbol{G}(h)\|_2\ \|\boldsymbol{K}^{-1}\|_2 \ |c(h)|_2 + \|\boldsymbol{K}^{-1}\|_2\ |c(h)-\mathfrak{C}(\alpha)|_2. \label{Conv_omega_off}
    \end{align} A Taylor expansion of $g$ around the origin gives \begin{align}\label{Conv_g_off}
        |1- g(\xi_k h)|\leq C h^{2p}, \quad k=1,\dots,N_p,
    \end{align} and hence \begin{align}
        \|\boldsymbol{I} - \boldsymbol{G}(h)\|_2 = \max\{|1-g(\xi_k h)|:k=1,\dots,N_p\} \lesssim h^{2p}.
    \end{align} From \cref{Conv_omega_off}, \cref{Conv_c(h)_off} and \cref{Conv_g_off} we therefore have \begin{align}\label{Conv_omega(h)_off}
        |\omega(h) - \bar{\omega}|_2\leq C h^{2p}.
    \end{align}
    
    We now prove the the accuracy of the corrected trapezoidal rule \cref{CorrectedTrapz_off_diag}. We denote the Taylor's polynomial of $\phi$ at the origin by \begin{align}
        P_{\phi}(x)=\sum_{\substack{\xi\in\mathbb{N}_0^2,\\ 0\leq |\xi|_1\leq 2p-1}}\frac{x^{\xi}}{\xi!}\frac{\partial^{\xi}\phi(0)}{\partial x^{\xi}},
    \end{align} and define $\tilde{\phi}=P_{\phi}\cdot g$. It is clear that \begin{align}\label{PhisTaylor_off}
    \partial^{k}[\phi(x)-\tilde{\phi}(x)]|_{x=0}=0,\quad k\in\mathbb{N}_0^2:|k|_1=0,1,\dots,2p-1.
\end{align} Writing $I[f]\coloneqq \int_{\mathbb{R}^2} f(x)dx$, we can split $Q_h^p[f]-I[f]$ as \begin{align*}
    Q_h^p[f]-I[f] & = A_h^p[\phi-\tilde{\phi}]+(T_h^0-I)[(\phi-\tilde{\phi})\cdot s_{12}]+ (Q_h^p-I)[\tilde{\phi}\cdot s_{12}] \\
    & \eqqcolon E_1+E_2+E_3 .
\end{align*} To estimate $E_1$, from \cref{PhisTaylor_off} and Taylor's theorem we have \begin{align*}
    |\phi(x)-\tilde{\phi}(x)|\leq \max_{\substack{x\in\mathbb{R}^2,\\ k\in\mathbb{N}_0^2,\\|k|_1=2p}}\left(\frac{|\partial^k(\phi - \tilde{\phi})(x)|}{k!}\right)\ |x|^{2p},
\end{align*} From definition of $A_h^p$ in \cref{CorrectedTrapz_off_diag} we obtain \begin{align}\label{E1Est_off}
    |E_1| = |A_h^p[\phi-\tilde{\phi}]| 
    \leq h^{2-\alpha}\max_{1\leq j\leq N_p}|\bar{\omega}_j|\max_{\xi\in\mathcal{G}_j:1\leq j\leq N_p}|\phi(\xi h)-\tilde{\phi}(\xi h)| 
    \leq h^{2p+2-\alpha}.
\end{align} To bound $E_2$, we apply \cref{corollary:order trapez off} to $\phi-\tilde{\phi}$, giving rise to \begin{align}\label{E2Est_off}
    |E_2| = |(T_h^0-I)[(\phi-\tilde{\phi})\cdot s_{12}]|\leq C h^{2p+2-\alpha}.
\end{align} Finally, it remains to estimate $E_3$. Since \begin{align}\label{PreE_3_off}
    |E_3|&= |Q^p_h[\tilde{\phi}\cdot s_{12}] - I[\tilde{\phi}\cdot s_{12}]| \nonumber \\
    &\leq \sum_{\substack{\xi\in\mathbb{N}_0^2,\\0\leq|\xi|_1\leq 2p-1}}\frac{1}{\xi!}\left|\frac{\partial^{\xi}\phi(0)}{\partial x^{\xi}}\right||Q_h^p[g\cdot s_{12}\cdot x^{\xi}] - I[g\cdot s_{12}\cdot x^{\xi}]| \nonumber \\
    &\leq \sum_{\substack{\xi\in\mathbb{N}_0^2,\\0\leq|\xi|_1\leq 2p-1}}|T_h^0[g\cdot s_{12}\cdot x^{\xi}] - I[g\cdot s_{12}\cdot x^{\xi}]+A^p_h[g\cdot x^{\xi}]|.
\end{align} We claim $T_h^0[g\cdot s_{12}\cdot x^{\xi}]$, $I[g\cdot s_{12}\cdot x^{\xi}]$ and $A^p_h[g\cdot x^{\xi}]$ vanish when at least one of $\xi_1$ or $\xi_2$ is even. To see this, by radial symmetry of $g$ and $\frac{1}{|x|^{2+\alpha}}$, we have \begin{align} \label{OddpowerI_off}
    &I[g\cdot s_{12}\cdot x^{\xi}]  = \int_{\mathbb{R}^2}\frac{g(x)}{|x|^{2+\alpha}}x^{\xi+e_1+e_2}\ \d x \nonumber \\
    & = \int_{[0,\infty)^2}\frac{g(x)}{|x|^{2+\alpha}}[x^{\xi+e_1+e_2}+(-x_1)^{\xi_1+1}(x_2)^{\xi_2+1}+(x_1)^{\xi_1+1}(-x_2)^{\xi_2+1}+ (-x_1)^{\xi_1+1}(-x_2)^{\xi_2+1}] \ \d x\nonumber \\
    & = \int_{[0,\infty)^2}g(x)s_{12}(x)x^{\xi}[1-(-1)^{\xi_1}-(-1)^{\xi_2}+(-1)^{|\xi|_1}]\ \d x.
\end{align} It is easy to see that \cref{OddpowerI_off} vanish when at least one of $\xi_1$ or $\xi_2$ is even. Argument for $T_h^0[g\cdot s_{12}\cdot x^{\xi}]=0$ under same condition for $\xi$ closely resembles above. Recall that $g(\xi h)$ is constant over each $\mathcal{G}_{q-m,m}$ and if $(a_1,a_2)\in\mathcal{G}_{q-m,m}$ then $(\pm a_1,\pm a_2),(\mp a_1,\pm a_2), (\pm a_2,\pm a_1), (\mp a_2,\pm a_1)\in\mathcal{G}_{q-m,m}$ for all $(q-m,m)\in\mathcal{I}_p$. Denote $g_{q-m,m} = g(\xi h),\ \text{for any }\xi\in\mathcal{G}_{q-m,m}$, we have \begin{align}
    A^p_h[g\cdot x^{\xi}] &= h^{2-\alpha}\sum_{(q-m,m)\in\mathcal{I}_p}\bar{\omega}_{q-m,m} g_{q-m,m} \sum_{\beta\in\mathcal{G}_{q-m,m}}\text{sgn}(\beta_1\beta_2)(\beta h)^{\xi}, \nonumber \\
    \sum_{\beta\in\mathcal{G}_{q-m,m}}\text{sgn}(\beta_1\beta_2)(\beta h)^{\xi} &= (q-m)^{\xi_1}m^{\xi_2}\Big(1-(-1)^{\xi_1}-(-1)^{\xi_2}+(-1)^{|\xi|_1}\Big) \nonumber\\
    & + (1-\delta_{q-m,m})m^{\xi_1}(q-m)^{\xi_2}\Big(1-(-1)^{\xi_1}-(-1)^{\xi_2}+(-1)^{|\xi|_1}\Big) \nonumber\\
    &= 0,
\end{align} if at least one of $\xi_1$ or $\xi_2$ is even. 

By another round of symmetry argument, it is also clear that for all $\xi\in\mathbb{N}_0^2$, \begin{align}
    I(g\cdot s_{12}\cdot x^{(\xi_1,\xi_2)}) & = I(g\cdot s_{12} \cdot x^{(\xi_2,\xi_1)}),\\
    T_h^0(g\cdot s_{12}\cdot x^{(\xi_1,\xi_2)}) & = T_h^0(g\cdot s_{12}\cdot x^{(\xi_2,\xi_1)}), \\
    A^p_h(g\cdot x^{(\xi_1,\xi_2)}) & = A^p_h(g\cdot x^{(\xi_2,\xi_1)}).
\end{align} When $\xi=0$, using the odd symmetry of $s_{12}$ and property of $A_h^p$, we have \begin{align}
    I(g\cdot s_{12}) & = T_h^0(g\cdot s_{12}) = 0, \\
    A_h^p(g) &= h^{2-\alpha}\sum_{(q-m,m)\in\mathcal{I}_p}\omega_{q-m,m} g_{q-m,m} \underbrace{\sum_{\beta\in\mathcal{G}_{q-m,m}}\text{sgn}(\beta_1\beta_2)}_{=0} = 0.
\end{align} Therefore, we can now re-write \cref{PreE_3_off} as \begin{align}\label{Pre2E3_off}
    |E_3| &\leq \sum_{|\xi|_1 = 0,2,\dots,2p-2}|T_h^0[g\cdot s_{12}\cdot x^{\xi}] - I[g\cdot s_{12}\cdot x^{\xi}]+A^p_h[g\cdot x^{\xi}]|\nonumber \\
    &= \sum_{\xi \in \mathcal{I}_p} |T_h^0[g\cdot s_{12}\cdot x^{2\xi-e_1-e_2}] - I[g\cdot s_{12}\cdot x^{2\xi-e_1-e_2}]+A^p_h[g\cdot x^{2\xi -e_1-e_2}]|.
\end{align} By \cref{LinearSysOff_Second} and \cref{Conv_omega(h)_off}, $E_3$ can be estimated by \begin{align}\label{E3Est_off}
    |E_3|&\leq Ch^{2-\alpha}\sum_{\xi\in\mathcal{I}_p}\sum_{(q-m,m)\in\mathcal{I}_p}|\omega_{q-m,m}(h)-\bar{\omega}_{q-m,m}|\left(\sum_{\beta\in\mathcal{G}_{q-m}^m}|g(\beta h)||\beta h|^{2|\xi|_1-e_1-e_2}\right) \nonumber \\
    & \leq Ch^{2-\alpha}|\omega(h) - \bar{\omega}|_2 \leq Ch^{2p+2-\alpha}.
\end{align}
Putting together \cref{E1Est_off}, \cref{E2Est_off} and \cref{E3Est_off}, we proved \begin{align}
     |Q_h^p[f] - I[f]|\leq C h^{2p+2-\alpha}.
\end{align}
\end{proof}

The following corollary provides a Taylor's expansion of $c(h)$ as in \cref{LinearSysOff_Second}. It is used in Richardson extrapolation for numerical computation of correction weights. Recall that $\mathfrak{C}(\alpha) = \lim_{h\to 0}c(h)$ in \cref{Const_C_off}. We skip the proof that runs almost identically with that of \cref{Taylor_of_c(h)_on}.
\begin{corollary}\label{Taylor_of_c(h)_off}
Let $p,k\in\mathbb{N}_0:p\geq 1$. For any $g\in C_c^\infty(\mathbb{R}^2)$ such that $g(0) = 1$, $\partial^lg(0) = 0,\ l\in\mathbb{N}_0^2$ with $0\not = |l|_1\leq 2p-1$, we have\begin{align}
    c(h) = \mathfrak{C}(\alpha) + h^{2p}\sum_{j=0}^k C_jh^{2j},
\end{align} where $C_j$'s are constants independent of $h$.
\end{corollary}

%%%%%%%%%%%%%%%%%%%%%%%%%%%%%%%%%%%%%%%%%%%%%%%%%%%%%%%%%%%%%%%%%%%%%%%%%%%%%%%%%%%%%%%%%%%%%%%%%%%%%%%%%%%%%%%%%%%%%%%%%%%%%%%%%%%%%%%%%%%%%%%%%%%%%%%%%%%
\section{Non-singularity of Coefficient Matrices \texorpdfstring{$\boldsymbol{K}$}{TEXT}} \label{sectionNSCM}
\subsection{The On-diagonal Case}
In this subsection we prove the on-diagonal coefficient matrix $\boldsymbol{K}$ defined in \cref{MatrixK} is non-singular. We order all $\xi\in\mathcal{I}_p$ as follows, $\xi\in\mathcal{I}_p$ with $\xi_1 = 0$ first, followed by all $\xi\in\mathcal{I}_p$ with $\xi_2=0$ and then followed by the rest $\xi\in\mathcal{I}_p$ with $\xi_1\xi_2 \not= 0$, where we denote the three subsets by $\mathcal{C}_{1},\mathcal{C}_{2},\mathcal{C}_{3}$ respectively. Namely, $\mathcal{C}_{1}$ is arranged by \begin{align}
    \mathcal{C}_{1} = \{(0,r):0\leq r\leq p\} 
        = \{(0,0),(0,1),\dots,(0,p)\},
\end{align} and similarly, \begin{align}
    \mathcal{C}_{2} &= \{(s,0):1\leq s\leq p\} 
        =\{(1,0),(2,0),\dots,(p,0)\}.
\end{align} The arrangement for $\mathcal{C}_{3}$ can in fact be arbitrary, thus we simply choose a natural one as follows, pictorially, we can list all elements in $\mathcal{C}_{3}$ as a triangle \begin{align}
    \begin{array}{cccccc}
        (1,1) & (1,2) & \dots & (1,k) & \cdots & (1,p-1) \\
         \downarrow & (k-1,2) \nearrow & \vdots & \nearrow\\
         (k,1)\nearrow & \dots & (k,p-k) & \\
         \vdots & \nearrow & & \\
         (p-1,1) & & &
    \end{array}\
\end{align} We arrange all elements in $\mathcal{C}_{3}$ from $(1,1)$, followed by diagonal $(2,1),(1,2)$, $\dots$, followed by diagonal $(k,1),\dots,(1,k)$, $\dots$, i.e. \begin{align}
    \mathcal{C}_{3} = \{(1,1),(2,1),(1,2),\dots,(p-1,1),(p-2,2),\dots,(1,p-1)\}.
\end{align} We also arrange $\mathcal{G}$ in the same way. We first show matrix $\boldsymbol{K}$ organized under this arrangement has a special structure. \begin{lemma}\label{StructureLemma}
Let $p\geq 1$, there exist an ordering on $\mathcal{I}_p$ such that \begin{enumerate}
    \item $K_{1,1} = 1$ while $K_{j,1} = 0$ for all $2\leq j\leq N_p$.
    \item $\boldsymbol{K}$ can be expressed as a block matrix of the form \begin{align}\label{Matrix_K_on}
        \boldsymbol{K} = \begin{pmatrix} \boldsymbol{A} & \boldsymbol{B} \\ \boldsymbol{C} & \boldsymbol{D} \end{pmatrix},
    \end{align} such that $\boldsymbol{A}\in \mathbb{R}^{(2p+1)\times (2p+1)}$, $\boldsymbol{C} = \boldsymbol{0}\in \mathbb{R}^{\frac{p(p-1)}{2}\times (2p+1)}$, $\boldsymbol{D}\in\mathbb{R}^{\frac{p(p-1)}{2}\times \frac{p(p-1)}{2}}$. As a result, \begin{align}
        \det \boldsymbol{K} = \det \boldsymbol{A} \cdot \det \boldsymbol{D}.
    \end{align}
\end{enumerate} Moreover, the block $\boldsymbol{A}$ is non-singular.
\end{lemma}
\begin{proof} We use the ordering on $\mathcal{I}_p$ introduced in the beginning of this subsection. It is easy to verify \begin{align*}
    K_{1,1} = 0^{2\cdot 0}\cdot 0^{2\cdot 0} = 1,
\end{align*} and for all $\xi_j\in\mathcal{I}_p\setminus\{0\}$, either $\xi_{j,1}\not =0$ or $\xi_{j,2}\not = 0$, then \begin{align}
    K_{j,1} = 0^{2\xi_{j,1}}\cdot 0^{2 \xi_{j,2}} = 0, \quad \forall \ 2\leq j\leq N_p.
\end{align} The first part of the lemma is proved.

It is obvious that $|\mathcal{C}_{1}|=p+1$, $|\mathcal{C}_{2}|=p$ and $|\mathcal{C}_{3}|=N_p-2p-1 = \frac{p(p-1)}{2}$. $\boldsymbol{A}$ is generated by all $\xi\in \mathcal{C}_{1}\cup \mathcal{C}_{2}$ and $\mathcal{G}_\gamma:\gamma\in \mathcal{C}_{1}\cup \mathcal{C}_{2}$, thus it has size $(2p+1)\times (2p+1)$. Similarly, $\boldsymbol{D}$ is generated by all $\xi\in \mathcal{C}_{3}$ and $\mathcal{G}_\gamma:\gamma\in \mathcal{C}_{3}$, hence its dimension is $\frac{p(p-1)}{2}\times \frac{p(p-1)}{2}$. To show $\boldsymbol{C}$ is a \emph{zero} matrix, we note that all elements in $\boldsymbol{C}$ are of the form 
$\sum_{\beta\in \mathcal{G}_{\gamma}}\beta^{2\xi},\quad \forall \xi\in \mathcal{C}_{3}$, $\forall\gamma\in \mathcal{C}_{1}\cup \mathcal{C}_{2}$.
If $\gamma\in \mathcal{C}_{1}$, then $\gamma = (0,r)$ for some $0\leq r\leq p$. Since $\xi_1$ and $\xi_2$ are both non-zeros, we have $\sum_{\beta\in \mathcal{G}_{\gamma}}\beta^{2\xi} = 2r^{2\xi_1}\cdot 0^{2\xi_2}=0$. The case for $\gamma\in \mathcal{C}_{2}$ also holds by identical argument.

We are left to show $\boldsymbol{A}$ is non-singular. We claim that $\boldsymbol{A}$ has the form \begin{align}
    \boldsymbol{A} = \begin{pmatrix}1 & * & * \\ 0_{p\times 1} & 2\boldsymbol{A}_1 & \boldsymbol{A}_2=\boldsymbol{0}_{p\times p} \\ 0_{p\times 1}& \boldsymbol{A}_3=\boldsymbol{0}_{p\times p} & 2\boldsymbol{A}_4 \end{pmatrix},
\end{align} where both $\boldsymbol{A}_1,\boldsymbol{A}_4\in\mathbb{R}^{p\times p}$ are Vandermonde matrices. To see this, all elements in $2\boldsymbol{A}_1$ can be expressed by \begin{align}
    \sum_{\beta\in\mathcal{G}_{\gamma}}\beta^{2\xi},\quad \text{where}\ \xi,\gamma\in \mathcal{C}_{1}\setminus\{(0,0)\},
\end{align} let $\xi=(0,r)$ and $\gamma=(0,s)$ for some $1\leq r,s\leq p$, so \begin{align}
    2\boldsymbol{A}_1(r,s)= \boldsymbol{A}(r+1,s+1)=\sum_{\beta\in\mathcal{G}_{\gamma}}\beta^{2\xi} = 2(s^2)^r,
\end{align} which shows $\boldsymbol{A}_1$ is a Vandermonde matrix. Repeating the above argument, elements in $2\boldsymbol{A}_4$ can be expressed by \begin{align}
    \sum_{\beta\in\mathcal{G}_{\gamma}}\beta^{2\xi},\quad \text{where}\ \xi,\gamma\in \mathcal{C}_{2},
\end{align} let $\xi=(r,0)$ and $\gamma=(s,0)$ for some $1\leq r,s\leq p$, we have\begin{align}
    2\boldsymbol{A}_4(r,s) = \boldsymbol{A}(p+1+r,p+1+s)= \sum_{\beta\in\mathcal{G}_{\gamma}}\beta^{2\xi} = 2(s^2)^r,
\end{align} which show $\boldsymbol{A}_4$ is a Vandermonde matrix. Finally, let $\xi=(0,r)\in \mathcal{C}_{1}$ and $\gamma=(s,0)\in \mathcal{C}_{2}$, we compute \begin{align}
    \sum_{\beta\in\mathcal{G}_{\gamma}}\beta^{2\xi} = 0.
\end{align} Thus $\boldsymbol{A}_3$ is the $p\times p$ zeros matrix. $\boldsymbol{A}_2 = \boldsymbol{0}_{p\times p}$ can then be proved analogously. The non-singularity of matrix $\boldsymbol{A}$ is now obvious.
\end{proof}

We now focus on showing the sub-matrix $\boldsymbol{D}$ in \cref{Matrix_K_on} is non-singular. We use algebraic combinatorial method to show the determinant of $\boldsymbol{D}$ is non-zero. To do this, let $x_i\in\mathbb{R}$, $1\leq i\leq p-1$, we define matrix $\boldsymbol{E}\in\mathbb{R}^{\frac{p(p-1)}{2}\times \frac{p(p-1)}{2}}$ as follows: for all $1\leq i,j\leq \frac{p(p-1)}{2}$ \begin{align} \label{MatrixE}
    E_{i,j} = x_{\xi_{j,1}}^{\xi_{i,1}}\cdot x_{\xi_{j,2}}^{\xi_{i,2}} = (x_{\xi_{j,1}}, x_{\xi_{j,2}})^{(\xi_{i,1}, \xi_{i,2})},
\end{align} where $\xi_i = (\xi_{i,1},\xi_{i,2})$ and $\xi_j = (\xi_{j,1},\xi_{j,2})$ are the $i$-th and $j$-th elements in $\mathcal{C}_{3}$, respectively. We now prove 
\begin{theorem}\label{StructureThm}
For all $p\geq 2$, \begin{align}
    \det \boldsymbol{E} = C\prod_{1\leq j\leq p-1}x_j^{2(p-j)}\prod_{1\leq i< j\leq p-1}(x_j-x_i)^{2(p-j)},
\end{align} where $C$ is a non-zero constant.
\end{theorem}
\begin{proof}
    Denote  $\det \boldsymbol{E} \eqqcolon F(x_1,\dots,x_{p-1})$. We observe that the multivariate polynomial $F$ is homogeneous, i.e., all terms in $F$ have the same degree. To see this, apply Leibniz's formula, we have \begin{align}
        F(x_1,\dots,x_{p-1}) = \sum_{\sigma\in S_{\frac{p(p-1)}{2}}} \text{sgn}(\sigma)\prod_{1\leq i\leq \frac{p(p-1)}{2}} E_{i,\sigma(i)},
    \end{align} where $S_m$ is the permutation group of the set $\{1,\dots,m\}$, $m\in\mathbb{N}$. By construction of $\boldsymbol{E}$, every elements in $i$-th row of $\boldsymbol{E}$ share the same degree, namely $\xi_{i,1}+\xi_{i,2}$, for any $1\leq i\leq \frac{p(p-1)}{2}$, so $\prod_{1\leq i\leq \frac{p(p-1)}{2}} K_{i,\sigma(i)}$ has same degree for all $\sigma\in S_{\frac{p(p-1)}{2}}$ and this implies homogeneity of $F$. We can then compute the degree of $F$ by \begin{align}
        \deg F = \sum_{q=2}^{p}q(q-1) = \frac{p^3-p}{3}.
    \end{align} 
    
    Consider $j\in\mathbb{N}$ such that $1\leq j\leq p-1$, we claim that $x_j^{2(p-j)}$ divides $F$. By construction of $\boldsymbol{E}$, there are $(p-j)$'s $x_j$ appearing in $(x_j,\cdot)$ and $(p-j)$'s $x_j$ appearing in $(\cdot,x_j)$, so there are $2(p-j)$'s $x_j$ in total. Hence the algebraic multiplicity of $x_j$ is at least $2(p-j)$.
    
    For arbitrary pair of integers $(i,j)$ such that $1\leq i<j\leq p-1$, if we substitute $x_j$ with $x_i$, then there would be at least two identical columns in $\boldsymbol{E}$ and hence $F$ will vanish. Furthermore, there are $(p-j)$'s $x_j$ appearing in $(x_j,\cdot)$ and $(p-j)$'s $x_j$ appearing in $(\cdot,x_j)$. As a result, there are $2(p-j)$ instances. Therefore, $(x_j-x_i)^{2(p-j)}$ divides $F$.
    
    Now, by unique factorization property of multivariate polynomial ring, there exists a multivariate polynomial $C(x_1,\dots,x_{p-1})$ such that \begin{align}
        F(x_1,\dots,x_{p-1}) &= C(x_1,\dots,x_{p-1})\prod_{1\leq j\leq p-1}x_j^{2(p-j)}\prod_{1\leq i< j\leq p-1}(x_j-x_i)^{2(p-j)}, \\
        & \eqqcolon C(x_1,\cdots,x_{p-1})H(x_1,\dots,x_{p-1}).
    \end{align} $C(x_1,\dots,x_{p-1})$ will be a constant polynomial if
    \begin{align}
        \deg H = \deg F.
    \end{align} We compute \begin{align}
        \deg H  = \sum_{j=1}^{p-1}2(p-j) + \sum_{j=2}^{p-1}\sum_{i=1}^{j-1}2(p-j) 
         = p(p-1)+\frac{1}{3}p(p-1)(p-2) 
        = \frac{p^3-p}{3} = \deg F.
    \end{align} This concludes the proof of  \cref{StructureThm}.
\end{proof} 

\begin{corollary} \label{K_is_nonsingular_on}
The coefficient matrix $\boldsymbol{K}$ as in \cref{MatrixK} is non-singular for all $p\geq 1$.
\end{corollary}
\begin{proof}
    By \cref{StructureLemma}, it suffices to show $\boldsymbol{D}$ is non-singular. Note that  $\boldsymbol{D} = 4\boldsymbol{E}$ where $\boldsymbol{E}$ defined in \cref{MatrixE} with $x_i=i^2$ for all $1\leq i\leq p-1$ and all $x_i$'s are non-zeros and mutually distinct. Hence by  \cref{StructureThm} \begin{align}
        \det \boldsymbol{D} = 2^{p(p-1)}\det \boldsymbol{E}\not = 0.
    \end{align} This completes the proof.
\end{proof}

\subsection{The Off-diagonal Case}
In this subsection, we will prove off-diagonal coefficient matrix $\boldsymbol{K}$ defined at \cref{MatrixK_12} is non-singular, therefore the linear system \cref{LinearSysOff_Second} always admits unique solution.

Let $p\in\mathbb{N},\ p\geq 2$. We recall the notations for $\mathcal{G}_{q-m,m}$ \cref{eq:G_off} and $\mathcal{I}_p$ \cref{I_p off} and define: \begin{align}
    \mathcal{G}  \coloneqq \left\{ \mathcal{G}_{q-m,m}: 2\leq q\leq p,\ 1\leq m \leq \left\lfloor\frac{q}{2}\right\rfloor\right\}.
\end{align} The value of $N_p\coloneqq|\mathcal{I}_p|= |\mathcal{G}|$ is given \begin{align} \label{eq:size of Ip off}
    N_p = \left\{\begin{array}{cc}
       \frac{p^2}{4}  &, \text{if } p \text{ is even} \\
         \frac{p^2-1}{4} &, \text{if } p \text{ is odd}
    \end{array}\right. .
\end{align} We arrange $\mathcal{I}_p$ first by the increasing order of magnitude of 1-norm and followed by, within each group of same 1-norm, the reverse dictionary order:\begin{align}
\mathcal{I}_p &= \left\{(1,1),(2,1),(3,1),(2,2), \dots,(p-1,1),(p-2,2),\dots,\left(\left\lceil\frac{p}{2}\right\rceil,\left\lfloor\frac{p}{2}\right\rfloor \right)\right\}\eqqcolon \{\beta_j\}_{1\leq j\leq N_p}.
\end{align} We arrange $\mathcal{G}$ in the same fashion with $\mathcal{I}_p$ so that \begin{align}
    \mathcal{G} & = \left\{\mathcal{G}_{1,1},\mathcal{G}_{2,1},\dots, \mathcal{G}_{p-1,1},\mathcal{G}_{p-2,2},\dots, \mathcal{G}_{\lceil\frac{p}{2}\rceil,\lfloor\frac{p}{2}\rfloor}\right\} \eqqcolon \{\mathcal{G}_j\}_{1\leq j\leq N_p}.
\end{align}

% \begin{lemma}\label{StructureLemma_off}
% For all $p\in\mathbb{N},\ p\geq 2$
% \begin{align}
%     \boldsymbol{K} = \boldsymbol{E}\boldsymbol{H},
% \end{align} where $\boldsymbol{H}\in\mathbb{R}^{N_p\times N_p}$ is a diagonal matrix so that for all $1\leq i\leq N_p$ \begin{align}
%     \boldsymbol{H}_{ii} = \left\{\begin{array}{ccc}
%         2|\beta_{i,1}\beta_{i,2}|, & \text{if} &|\beta_{i,1}| = |\beta_{i,2}|, \ (\beta_{i,1},\beta_{i,2})\in \mathcal{G}_i \\
%         4|\beta_{i,1}\beta_{i,2}|, & \text{if} & \text{otherwise}
%     \end{array}\right.
% \end{align} and for all $1\leq i,j\leq N_p$ \begin{align}
%     \boldsymbol{E}_{i,j} = \beta_{j,1}^{2\xi_{i,1}}\beta_{j,2}^{2\xi_{i,2}} + \beta_{j,2}^{2\xi_{i,1}}\beta_{j,1}^{2\xi_{i,2}},
% \end{align} here, $\xi_i$ is $i$-th element in $\mathcal{M}_p$.
% \end{lemma}
The following lemma reveals the structure of the coefficient matrix $\boldsymbol{K}$ given in \cref{MatrixK_12}.
\begin{lemma}\label{StructureLemma_off}
For all $p\in\mathbb{N},\ p\geq 2$
\begin{align}
    \boldsymbol{K} = \boldsymbol{E}\boldsymbol{H},
\end{align} where $\boldsymbol{H}\in\mathbb{R}^{N_p\times N_p}$ is a diagonal matrix with \begin{align}
    \boldsymbol{H}_{jj} = \left\{\begin{array}{ccc}
        \sfrac{2}{|\xi_{j,1}\xi_{j,2}|}, & \text{if} &|\xi_{j,1}| = |\xi_{j,2}|, \ \forall \xi_j = (\xi_{j,1},\xi_{j,2})\in \mathcal{G}_j \\
        \sfrac{4}{|\xi_{j,1}\xi_{j,2}|}, & \text{if} & \text{otherwise},
    \end{array}\right.
\end{align} and $\boldsymbol{E}$ is given by \begin{align}
    \boldsymbol{E}_{i,j} = |\xi_{j,1}|^{2\xi_{i,1}}|\xi_{j,2}|^{2\xi_{i,2}} + |\xi_{j,2}|^{2\xi_{i,1}}|\xi_{j,1}|^{2\xi_{i,2}},
\end{align} for all $1\leq i,j\leq N_p$. Here, $\xi_i,\ \xi_j$ is $i$-th and $j$-th element in $\mathcal{I}_p$, respectively.
\end{lemma}
\begin{proof}
From the definition of coefficient matrix $\boldsymbol{K}$ in \cref{MatrixK_12}, for any $1\leq i,j\leq N_p$ we have \begin{align}
    K_{i,j} &= \sum_{\xi_j\in\mathcal{G}_j}\text{sgn}(\xi_{j,1}\xi_{j,2})\xi_j^{2\xi_i-(1,1)} \nonumber \\
    & = \left\{\begin{array}{ccc}
        4\sfrac{(\xi_{j,1},\xi_{j,2})^{2\xi_i}}{|\xi_{j,1}\xi_{j,2}|}, & \text{if} &|\xi_{j,1}| = |\xi_{j,2}|, \\
        4\sfrac{[(\xi_{j,1},\xi_{j,2})^{2\xi_i} + (\xi_{j,2},\xi_{j,1})^{2\xi_i}]}{|\xi_{j,1}\xi_{j,2}|}, & \text{if} & \text{otherwise}.
    \end{array}\right.
\end{align} The lemma follows immediately.
\end{proof}

From \cref{StructureLemma_off} we know that $\det{\boldsymbol{K}} = \det{\boldsymbol{E}}\det{\boldsymbol{H}}$ and it is clear that $\boldsymbol{H}$ is non-singular, therefore non-singularity of $\boldsymbol{K}$ follows once we show $\boldsymbol{E}$ is non-singular. We prove this fact by a algebraic combinatorial argument. Let $x_i\in\mathbb{R},1\leq i\leq p-1$, we define matrix $\boldsymbol{T}(x_1,\dots,x_{p-1})\in\mathbb{R}^{N_p\times N_p}$ by \begin{align}
    T_{i,j}  = x_{\xi_{j,1}}^{\xi_{i,1}}x_{\xi_{j,2}}^{\xi_{i,2}} + x_{\xi_{j,1}}^{\xi_{i,2}}x_{\xi_{j,2}}^{\xi_{i,1}}
     = (x_{\xi_{j,1}}, x_{\xi_{j,2}})^{(\xi_{i,1},\xi_{i,2})} + (x_{\xi_{j,2}}, x_{\xi_{j,1}})^{(\xi_{i,1},\xi_{i,2})},
\end{align} where $\xi_i = (\xi_{i,1},\xi_{i,2})$ and $\xi_j = (\xi_{j,1},\xi_{j,2})$ are the $i$-th and $j$-th element in $\mathcal{I}_p$, respectively.

\begin{theorem}\label{StructureThm_off}
Let $p\in\mathbb{N},\ p\geq 2$ \begin{align}
    \det \boldsymbol{T} = C\prod_{1\leq j\leq \lfloor\frac{p}{2}\rfloor}x_j^{p+1-j}\prod_{\lfloor\frac{p}{2}\rfloor<j\leq p-1}x_j^{p-j}\prod_{1\leq i<j\leq \lfloor\frac{p}{2}\rfloor}(x_j-x_i)^{p+1-j}\prod_{\substack{1\leq i<j,\\ \lfloor\frac{p}{2}\rfloor<j\leq p-1}}(x_j-x_i)^{p-j},
\end{align} where $C$ is a non-zero constant.
\end{theorem}
\begin{proof}
    Let us prove when $p = 2k$, $k\geq 1$, the proof for the $p$ odd is very similar. From \cref{eq:size of Ip off} we note that $N_{2k} = k^2$. Denote $\det \boldsymbol{T}\eqqcolon F(x_1,\dots,x_{2k-1})$. We observe that the multivariate polynomial $F$ is homogeneous. To see this, apply Leibniz's formula, we have \begin{align}
        F(x_1,\dots,x_{2k-1}) = \sum_{\sigma\in S_{k^2}} \text{sgn}(\sigma)\prod_{1\leq j\leq k^2} T_{j,\sigma(j)},
    \end{align} where $S_n$ is the permutation group of the set $\{1,\dots,n\}$, $n\in\mathbb{N}$. By the construction of $\boldsymbol{T}$, all elements in $i$-th row of $\boldsymbol{T}$ share the degree $|\xi_i|_1$, where $\xi_i \in\mathcal{I}_p$ for any $1\leq i\leq k^2$. Thus, $\prod_{1\leq j\leq k^2} \boldsymbol{T}_{j,\sigma(j)}$ has same degree for all $\sigma\in S_{k^2}$ and this implies homogeneity of $T$. We can then compute the degree of $F$ by \begin{align}
        \deg F & = \sum_{q=2}^{2k} q\lfloor\sfrac{q}{2}\rfloor = \frac{4}{3}k^3+\frac{1}{2}k^2+\frac{1}{6}k.
    \end{align} 
    
    Consider $j\in\mathbb{N}$ such that $1\leq j\leq 2k-1$, suppose we substitute $0$ into $x_j$, then there would be at least one \emph{zero}-column in $\boldsymbol{T}$ and hence $F$ would be zero. Therefore, each $x_j$ divides $F$. We further prove that, if $1\leq j\leq k$, then $x_j^{2k+1-j}$ divides $F$ and if $k< j\leq 2k-1$, then $x_j^{2k-j}$ divides $F$. By construction of $\boldsymbol{T}$, for all $1\leq j\leq k$, there are $j$'s $x_j$ appearing in $(x_j,\cdot)$ and $2k-2j+1$'s $x_j$ appearing in $(\cdot,x_j)$, thus there are $2k-j+1$'s $x_j$ in total. On the other hand, for all $k< j\leq 2k-1$, there are $2k-j$'s $x_j$ appearing in $(x_j,\cdot)$ and there is no $(\cdot,x_j)$, so there are $2k-j$'s $x_j$ in total. Hence the algebraic multiplicity of $x_j$ is at least $2k+1-j$ for all $1\leq j\leq k$ and $2k-j$ for all $k<j\leq 2k-1$.
    
    For arbitrary pair of integers $(i,j)$ such that $1\leq i<j\leq 2k-1$, if we substituted $x_j$ with $x_i$, then there would be at least two identical columns in $\boldsymbol{T}$ and hence $F$ would vanish. Furthermore, for all integer pair $(i,j)$ such that $1\leq i<j\leq k$, there are $j$'s $x_j$ appearing in $(x_j,\cdot)$ and $2k-2j+1$'s $x_j$ appearing in $(\cdot,x_j)$, and for all integer pair $(i,j)$ such that $1\leq i < j, k<j\leq 2k-1$, there are $2k-j$'s $x_j$ appearing in $(x_j,\cdot)$. As a result, there are $2k-j+1$'s $x_j$ for all $2\leq j \leq k$ and $2k-j$'s $x_j$ for all $k<j\leq 2k-1$. Therefore, $(x_j-x_i)^{2k-j+1}$ divides $F$ if $2\leq j \leq k$ and $(x_j-x_i)^{2k-j}$ divides $F$ if $k< j \leq 2k-1$.
    
    Now, by the unique factorization property of multivariate polynomial ring, there exists a multivariate polynomial $C(x_1,\dots,x_{2k-1})$ such that \begin{align}
        F(x_1\dots,x_{2k-1}) &= C(x_1,\dots,x_{2k-1})\prod_{1\leq j\leq k}x_j^{2k-j+1}\prod_{k<j\leq 2k-1}x_j^{2k-j} \nonumber \\
        &\prod_{1\leq i< j\leq k}(x_j-x_i)^{2k-j+1}\prod_{\substack{1\leq i< j,\\ k<j\leq 2k-1}}(x_j-x_i)^{2k-j}, \\
        & \eqqcolon C(x_1,\cdots,x_{2k-1})H(x_1,\cdots,x_{2k-1}). \nonumber
    \end{align} $C(x_1,\cdots,x_{2k-1})$ will be a constant if we can prove \begin{align}
        \deg H = \deg F.
    \end{align} We compute \begin{align}
        \deg H & = \sum_{j=1}^k (2k-j+1) + \sum_{j=k+1}^{2k-1}(2k-j) 
         + \sum_{j=2}^{k}\sum_{i=1}^{j-1}(2k-j+1) + \sum_{j=k+1}^{2k-1}\sum_{i=1}^{j-1}(2k-j) \nonumber \\
        & = \left(\frac{3k^2+k}{2}\right)+\left(\frac{k^2-k}{2}\right)
         +\left(\frac{2}{3}k^3-\frac{1}{2}k^2-\frac{1}{6}k\right)+\left(\frac{2}{3}k^3-k^2+\frac{1}{3}k\right) \nonumber \\
        & = \frac{4}{3}k^3+\frac{1}{2}k^2+\frac{1}{6}k = \deg F.
    \end{align} This concludes the proof.
\end{proof}
We are now ready to prove the non-singularity of the coefficient matrix $\boldsymbol{K}$ as in \cref{MatrixK_12}.
\begin{corollary}\label{K_is_non_singular_off}
Let $p\in\mathbb{N},\ p\geq 2$. The matrix $\boldsymbol{K}$ define in \cref{MatrixK_12} is non-singular. As a result, the linear system \cref{LinearSysOff_Second} always admits a unique solution.
\end{corollary}
\begin{proof}
    We only need to prove $\boldsymbol{E}$ is non-singular, according to \cref{StructureLemma_off}. Letting $x_i = i^2,\ 1\leq i\leq p-1$, we have \begin{align}
        \boldsymbol{E} = \boldsymbol{F}(1^2,\dots,(p-1)^2).
    \end{align} Since all $x_i$'s are non-zeros and mutually distinct, therefore  \cref{StructureThm_off} implies $\boldsymbol{E}$ is non-singular and the proof is complete.
\end{proof}
\begin{remark}\label{Rmk_on_MarinTornburg_prob}
By combining \cref{StructureThm} and \cref{K_is_non_singular_off}, with a few modifications one can show the linear systems (2.13) in \cite{MarinTornberg2014} corresponding to 2D weakly singular integral with singular kernel $s(x) = \sfrac{1}{|x|}$ is non-singular for arbitrary $p$. Consequently, this leads to a unconditional proof of Theorem 3.7 in \cite{MarinTornberg2014}.
\end{remark}

%%%%%%%%%%%%%%%%%%%%%%%%%%%%%%%%%%%%%%%%%%%%%%%%%%%%%%%%%%%%%%%%%%%%%%%%%%%%%%%%%%%%%%%%%%%%%%%%%%%%%%%%%%%%%%%%%%%%%%%%%%%%%%%%%%%%%%%%%%%
\section{Numerical Results} \label{sectionNR}
\subsection{Computation of Correction Weights}
In order to use the modified trapezoidal rules \cref{FirstCorrectedTrapezRule,,CorrectedTrapz_off_diag}, one needs accurate values of weights $\bar{\omega}_\beta$'s. In this section, we describe numerical method for computing $\bar{\omega}_\beta$'s.
\subsubsection{On-diagonal case}
We present the computation of the correction weights for the grids $\mathcal{I}_0$, $\mathcal{I}_1$ and $\mathcal{I}_2$ in \cref{Ip_on}. We solve the linear system \cref{MatrixLinearSystem3} in the proof of \cref{MainTheorem} to compute correction weight. Compared with the linear system \cref{MatrixLinearSystem}, it has the advantage of evaluating less limits as $h\to 0$. To find out $\mathfrak{C}(\alpha)$, which is the RHS of the system \cref{MatrixLinearSystem3}, we need $c(h)$ in \cref{RHSofLinearSystem}. We choose a radially symmetric, practically compactly supported function \begin{align}
    g(x) = \exp(-|x|^k), \label{eq:pract C_c}
\end{align} with $k = 6$. The reason that we choose $g$ as in \cref{eq:pract C_c}, rather than truly compactly supported one is that we can analytically evaluate the integral \begin{align*}
    \int_{\mathbb{R}^2\setminus\{0\}}g(x)s(x)x^{2\xi}\ \d x = \frac{2\Gamma(1.5 + \xi_1)\Gamma(0.5+\xi_2)\Gamma(\sfrac{(2-\alpha+2\xi_1+2\xi_2)}{k})}{k\Gamma(2+\xi_1+\xi_2)},
\end{align*} for all $0<\alpha<2$ and $\xi\in\mathcal{I}_2$. Combining \cref{MatrixLinearSystem3} with \cref{Taylor_of_c(h)_on} we have \begin{align}
    \boldsymbol{K}\bar{\omega} = \mathfrak{C}(\alpha) = c(h) + h^{6}\sum_{j=0}^{J}C_j h^{2j}. \label{eq:Compute weights by Richardson}
\end{align} \Cref{eq:Compute weights by Richardson} enables us to apply twice the Richardson extrapolation to find out $\mathfrak{C}(\alpha)$:\begin{align*}
    c_1(h) = c(\sfrac{h}{2}) + \frac{c(\sfrac{h}{2}) - c(h)}{4^{3}-1},\quad c_2(h) = c_1(\sfrac{h}{2}) + \frac{c_1(\sfrac{h}{2}) - c_1(h)}{4^{4}-1}.
\end{align*} We obtain the weights with 20 correct digits by solving $\bar{\omega} \approx \boldsymbol{K}^{-1}c_2(h)$ and ensuring $|c_2(h) - c_2(\sfrac{h}{2})|<10^{-21}$ for some small enough $h$. In practice, we found $h = \frac{1}{32}$ sufficiently small to give more than 20 correct digits. \cref{CorrectionWeightsDiagonal1} and \cref{CorrectionWeightsDiagonal2} present the correction weights corresponding to the singular kernel $s$ in \cref{SingularFunctionGeneral} when $\alpha = 0.5,1.5$ and $p=0,1,2$, respectively.
%%%%%%%%%%%%%%%%%%%%%%%%%%%%%%%%%%%%%%%%%
\begin{table}[h]
\caption{The on-diagonal correction weights for $\alpha = 0.5$.}
\begin{center}
\begin{tabular}{ c|c|c } 
\hline\hline
 $p$ & Grid points & Correction weights \\
  \hline
   0 & $\mathcal{G}_{0,0} = \{(0,0)\}$ & $\omega_{0,0} = 9.608446105899650591\times 10^{-1}$ \\
  \hline
   \multirow{3}{*}{1} & $\mathcal{G}_{0,0} = \{(0,0)\}$ & $\omega_{0,0} = 9.2275199269460481567\times 10^{-1}$ \\ 
   & $\mathcal{G}_{1,0} = \{(\pm 1,0)\}$ & $\omega_{1,0} = -3.8305792599451481531\times 10^{-2}$ \\ 
   & $\mathcal{G}_{0,1} = \{(0,\pm 1)\}$ & $\omega_{0,1} = 5.7352101547131603247\times 10^{-2}$\\ 
  \hline
  \multirow{6}{*}{2}   & $\mathcal{G}_{0,0} = \{(0,0)\}$ & $\omega_{0,0} = 0.91354757991861649779\times 10^{-1}$ \\ 
  & $\mathcal{G}_{1,0} = \{(\pm 1,0)\}$ & $\omega_{1,0} = -4.9714459296827069288\times 10^{-2}$ \\ 
   & $\mathcal{G}_{0,1} = \{(0,\pm 1)\}$ & $\omega_{0,1} = 7.3324618127490001511\times 10^{-2}$ \\ 
    & $\mathcal{G}_{2,0} = \{(\pm 2, 0)\}$ & $\omega_{2,0} = 2.2625071864653714109\times 10^{-3}$ \\
   &$\mathcal{G}_{1,1} = \{(\pm 1,\pm 1),(\mp 1,\pm 1)\}$ & $\omega_{1,1} = 1.1793189757570510571\times 10^{-3}$ \\
   & $\mathcal{G}_{0,2} = \{(0,\pm 2)\}$ & $\omega_{0,2} = -4.5827886329681250944\times 10^{-3}$ \\
%   \hline
%   \multirow{6}{2em}{Three} &  & $\mathcal{G}_{0,0} = \{(0,0)\}$ & $\omega_{0,0} = 0.227321038095332$ \\ 
%   &  & $\mathcal{G}_{1,0} = \{(\pm 1,0)\}$ & $\omega_{1,0} = -0.026863138268806$ \\ 
%   &  & $\mathcal{G}_{0,1} = \{(0,\pm 1)\}$ & $\omega_{0,1} = 0.039699349707528$ \\ 
%   &  & $\mathcal{G}_{2,0} = \{(\pm 2, 0)\}$ & $\omega_{2,0} = 0.000896132305413$ \\
%   & &$\mathcal{G}_{1,1} = \{(\pm 1,\pm 1)\}$ & $\omega_{1,1} = 0.002087564615371$ \\
%   & 3 & $\mathcal{G}_{0,2} = \{(0,\pm 2)\}$ & $\omega_{0,2} = -0.002877434380253$ \\
%  &  & $\mathcal{G}_{3,0} = \{(\pm 3,0)\}$ & $\omega_{3,0} = -0.000173315091308$ \\ 
%  &  & $\mathcal{G}_{2,1} = \{(\pm 2,\pm 1)\}$ & $\omega_{2,1} = 0.001275016062908$ \\ 
%  &  & $\mathcal{G}_{1,2} = \{(\pm 1,\pm 2)\}$ & $\omega_{1,2} = -0.001502080302256$ \\ 
%  &  & $\mathcal{G}_{0,3} = \{(0,\pm 3)\}$ & $\omega_{0,3} = 0.000348019903577$ \\ 
 \hline\hline
\end{tabular}
\end{center}
\label{CorrectionWeightsDiagonal1}
\end{table}
%%%%%%%%%%%%%%%%%%%%%%%%%%%%%%%%%%%%%%%%%%%%%%%%
\begin{table}[h]
\caption{The on-diagonal correction weights for $\alpha = 1.5$.}
\begin{center}
\begin{tabular}{ c|c|c } 
\hline\hline
 $p$ & Grid points & Correction weights \\
  \hline
   0 & $\mathcal{G}_{0,0} = \{(0,0)\}$ & $\omega_{0,0} = 5.0387797393965760507\times 10^{0}$ \\
  \hline
  \multirow{3}{*}{1}   & $\mathcal{G}_{0,0} = \{(0,0)\}$ & $\omega_{0,0} = 4.7857569346819649328\times 10^{0}$ \\ 
   & $\mathcal{G}_{1,0} = \{(\pm 1,0)\}$ & $\omega_{1,0} = 1.0971059048869895449\times 10^{-2}$ \\ 
    & $\mathcal{G}_{0,1} = \{(0,\pm 1)\}$ & $\omega_{0,1} = 1.1554034330843566347\times 10^{-1}$\\ 
  \hline
  \multirow{6}{*}{2}   & $\mathcal{G}_{0,0} = \{(0,0)\}$ & $\omega_{0,0} = 4.7305900462046469972\times 10^{0}$ \\ 
    & $\mathcal{G}_{1,0} = \{(\pm 1,0)\}$ & $\omega_{1,0} = 1.7018648395611181367\times 10^{-2}$ \\ 
   & $\mathcal{G}_{0,1} = \{(0,\pm 1)\}$ & $\omega_{0,1} = 1.3848756814856511801\times 10^{-1}$ \\ 
    & $\mathcal{G}_{2,0} = \{(\pm 2, 0)\}$ & $\omega_{2,0} = -4.4305641359382777203\times 10^{-3}$ \\
   &$\mathcal{G}_{1,1} = \{(\pm 1,\pm 1),(\mp 1,\pm 1)\}$ & $\omega_{1,1} = 5.8373335985059124819\times 10^{-3}$ \\
   & $\mathcal{G}_{0,2} = \{(0,\pm 2)\}$ & $\omega_{0,2} = -8.6554730092853198753\times 10^{-3}$ \\
%   \hline
%   \multirow{6}{2em}{Three} &  & $\mathcal{G}_{0,0} = \{(0,0)\}$ & $\omega_{0,0} = 1.176468597051447$ \\ 
%   &  & $\mathcal{G}_{1,0} = \{(\pm 1,0)\}$ & $\omega_{1,0} = 0.011233380317102$ \\ 
%   &  & $\mathcal{G}_{0,1} = \{(0,\pm 1)\}$ & $\omega_{0,1} = 0.073239949695416$ \\ 
%   &  & $\mathcal{G}_{2,0} = \{(\pm 2, 0)\}$ & $\omega_{2,0} = -0.0052267374005979$ \\
%   & &$\mathcal{G}_{1,1} = \{(\pm 1,\pm 1)\}$ & $\omega_{1,1} = 0.009914373606087$ \\
%   & 3 & $\mathcal{G}_{0,2} = \{(0,\pm 2)\}$ & $\omega_{0,2} = -0.005839037856640$ \\
%  &  & $\mathcal{G}_{3,0} = \{(\pm 3,0)\}$ & $\omega_{3,0} = 0.000307820088666$ \\ 
%  &  & $\mathcal{G}_{2,1} = \{(\pm \mp 2,\pm 1)\}$ & $\omega_{2,1} = 0.001164534807023$ \\ 
%  &  & $\mathcal{G}_{1,2} = \{(\pm \mp 1,\pm 2)\}$ & $\omega_{1,2} = -0.002183794820310$ \\ 
%  &  & $\mathcal{G}_{0,3} = \{(0,\pm 3)\}$ & $\omega_{0,3} = 0.000615849361944$ \\ 
 \hline\hline
\end{tabular}
\end{center}
\label{CorrectionWeightsDiagonal2}
\end{table}
%%%%%%%%%%%%%%%%%%%%%%%%%%%%%%%%%%%%%%%%%%%%%%%%%%%%%%%%%%%%%%%%%%%%%%%%%%%%%%%%%%%%%%%%%%%%%%%%%%%%%%%%%%%%%%%%%%%%%%%%%%%%%%%%%%%%%%%%%%%%%%%%%%%%%%%%%
\subsubsection{Off-diagonal Case}
We present correction weights for grids $\mathcal{I}_2$, $\mathcal{I}_3$ and $\mathcal{I}_4$ in \cref{I_p off}. In off-diagonal case, the procedure for computing correction weights is similar with that of the on-diagonal case described above. We choose $g$ as in \cref{eq:pract C_c} with $k = 8$. The analytical values of the integrals are \begin{align*}
    \int_{\mathbb{R}^2\setminus\{0\}}g(x)s(x)x^{2\xi-(1,1)}\ \d x = \frac{2\Gamma(0.5 + \xi_1)\Gamma(0.5+\xi_2)\Gamma(\sfrac{(2\xi_1+2\xi_2-\alpha)}{k})}{k\Gamma(1+\xi_1+\xi_2)},
\end{align*} for all $0<\alpha<2$ and $\xi\in\mathcal{I}_4$. \Cref{CorrectionWeightsOff_0.5} and \cref{CorrectionWeightsOff_1.5} present the correction weights corresponding to singular kernel in \cref{eq:off singular kernel} for $\alpha = 0.5,1.5$ and $p=2,3,4$, respectively.

%%%%%%%%%%%%%%%%%%%%%%%%%%%%%%%%%%%%%
\begin{table}
\caption{The off-diagonal corrections weights for $\alpha = 0.5$.}
\begin{center}
\begin{tabular}{ c|c|c } 
\hline\hline
 $p$ & Discretization points & Correction weights \\
  \hline
   2 & $\mathcal{G}_{1,1} = \{(\pm 1,\pm 1),(\mp 1,\pm 1)\}$ & $\omega_{1,1} = 0.0286760507735658016236634025724$ \\
  \hline
  \multirow{3}{*}{3}  & $\mathcal{G}_{1,1} = \{(\pm 1,\pm 1),(\mp 1,\pm 1)\}$ & $\omega_{1,1} = 0.0470072053054383020013851917611$ \\ 
    & $\mathcal{G}_{2,1} = \{(\pm \mp 2,\pm \pm 1), (\pm \mp 1,\pm \pm 2)\}$ & $\omega_{2,1} = -0.00458278863296812509443044729718$ \\ 
  \hline
  \multirow{4}{*}{4}   & $\mathcal{G}_{1,1} = \{(\pm 1,\pm 1),(\mp 1,\pm 1)\}$ & $\omega_{1,1} = 0.058498692309201978109$ \\ 
   & $\mathcal{G}_{2,1} = \{(\pm \mp2,\pm \pm1),(\pm \mp1,\pm \pm2)\}$ & $\omega_{2,1} = -0.0092844902620645196084$ \\ 
    & $\mathcal{G}_{3,1} = \{(\pm \mp3,\pm \pm1),(\pm \mp1,\pm \pm3)\}$ & $\omega_{3,1} = 0.0010440418727854435399$ \\ 
    & $\mathcal{G}_{2,2} = \{(\pm 2, \pm 2),(\mp 2, \pm 2)\}$ & $\omega_{2,2} = 0.00026276706897731017725$ \\
%   \hline
%   \multirow{6}{2em}{Four} &  & $\mathcal{G}_{1,1} = \{(\pm 1,\pm 1)\}$ & $\omega_{1,1} = 0.033097109392967$ \\ 
%   &  & $\mathcal{G}_{2,1} = \{(\pm 2,\pm 1),(\pm 1,\pm 2)\}$ & $\omega_{2,1} = -0.0132893940961497$ \\ 
%   & 5 & $\mathcal{G}_{3,1} = \{(\pm 3,\pm 1),(\pm 1,\pm 3)\}$ & $\omega_{3,1} = 0.002573357555047$ \\ 
%   &  & $\mathcal{G}_{2,2} = \{(\pm 2, \pm 2)\}$ & $\omega_{2,2} = 0.000340902341204$ \\
%   & &$\mathcal{G}_{4,1} = \{(\pm 4,\pm 1),(\pm 1,\pm 4)\}$ & $\omega_{4,1} = -0.000237424251011$ \\
%   &  & $\mathcal{G}_{3,2} = \{(\pm 3,\pm 2),(\pm 2,\pm 3)\}$ & $\omega_{3,2} = -0.000052381497173$ \\
 \hline\hline
\end{tabular}
\end{center}
\label{CorrectionWeightsOff_0.5}
\end{table}

%%%%%%%%%%%%%%%%%%%%%%%%%%%%%%
\begin{table}
\caption{The off-diagonal corrections weights for $\alpha = 1.5$.}
\begin{center}
\begin{tabular}{ c|c|c } 
\hline\hline
 $p$ & Discretization points & Correction weights \\
  \hline
   2 & $\mathcal{G}_{1,1} = \{(\pm 1,\pm 1)\}$ & $\omega_{1,1} = 0.0577701716542178317339761161235$ \\
  \hline
  \multirow{2}{*}{3}  & $\mathcal{G}_{1,1} = \{(\pm \mp 1,\pm 1)\}$ & $\omega_{1,1} = 0.0923920636913591112353501723599$ \\ 
    & $\mathcal{G}_{2,1} = \{(\pm 2,\pm 1), (\pm 1,\pm 2)\}$ & $\omega_{2,1} = -0.0086554730092853198753435140591$ \\ 
  \hline
  \multirow{4}{*}{4}  & $\mathcal{G}_{1,1} = \{(\pm 1,\pm 1)\}$ & $\omega_{1,1} = 0.11372612810258708544$ \\ 
   & $\mathcal{G}_{2,1} = \{(\pm 2,\pm 1),(\pm 1,\pm 2)\}$ & $\omega_{2,1} = -0.017474957624915655234$ \\ 
    & $\mathcal{G}_{3,1} = \{(\pm 3,\pm 1),(\pm 1,\pm 3)\}$ & $\omega_{3,1} = 0.0018475475899836517452$ \\ 
    & $\mathcal{G}_{2,2} = \{(\pm 2, \pm 2),(\mp 2, \pm 2)\}$ & $\omega_{2,2} = 0.00071464712784786418872$ \\
%   \hline
%   \multirow{6}{2em}{Four} &  & $\mathcal{G}_{1,1} = \{(\pm 1,\pm 1)\}$ & $\omega_{1,1} = 0.033097109392967$ \\ 
%   &  & $\mathcal{G}_{2,1} = \{(\pm 2,\pm 1),(\pm 1,\pm 2)\}$ & $\omega_{2,1} = -0.0132893940961497$ \\ 
%   & 5 & $\mathcal{G}_{3,1} = \{(\pm 3,\pm 1),(\pm 1,\pm 3)\}$ & $\omega_{3,1} = 0.002573357555047$ \\ 
%   &  & $\mathcal{G}_{2,2} = \{(\pm 2, \pm 2)\}$ & $\omega_{2,2} = 0.000340902341204$ \\
%   & &$\mathcal{G}_{4,1} = \{(\pm 4,\pm 1),(\pm 1,\pm 4)\}$ & $\omega_{4,1} = -0.000237424251011$ \\
%   &  & $\mathcal{G}_{3,2} = \{(\pm 3,\pm 2),(\pm 2,\pm 3)\}$ & $\omega_{3,2} = -0.000052381497173$ \\
 \hline\hline
\end{tabular}
\end{center}
\label{CorrectionWeightsOff_1.5}
\end{table}
%%%%%%%%%%%%%%%%%%%%%%%%%%%%%%%%%%%%%%%%%%%%%%%%%%%%%%%%%%%%%%%%%%%%%%%%%%%%%%%%%%%%%%%%%%%%%%%%%%%%%%%%%%%%%%%%%%%%%%%%%%%%%%%%%%%%%%%%%%%%%%%%%%%%%%%%%
\subsection{Order of Convergence of \texorpdfstring{$Q_h^p$}{TEXT}}
\subsubsection{On-diagonal case}
We verify the order of convergence of the corrected trapezoidal rule \cref{SecondCorrectedTrapezRule} for the on-diagonal weakly singular integral $I_{11}$ in \cref{singularint}. We choose a compactly supported function \begin{align}
    \phi(x) &= (1+x_1 + x_1^2)(1+x_2+x_2^2)\big((1-x_1^2)_+(1-x_2^2)_+\big)^7, 
\end{align} where $x_+\coloneqq \max\{x,0\}$. Note that $\phi\in C_c^6(\mathbb{R}^2)$ such that $\partial^k\phi(0)\not = 0,\ \forall k\in\mathbb{N}_0^2:|k|_1\leq 6$ and $supp(\phi) = [-1,1]^2$. We use $\phi$ to test \cref{MainTheorem} for $p=0,1,2$. Note that due to limited smoothness of this $\phi$, \cref{MainTheorem} guarantees the order of convergence $2p+4-\alpha$ for $p=0,1$, but is inconclusive for $p=2$. However we still observe the high order of accuracy. This implies there is room for further relaxation on smoothness criteria on $\phi$. The analytical integral $I = \int_{\mathbb{R}^2\setminus\{0\}}\phi(x)s(x)\ \d x$ for arbitrary $0<\alpha<2$ can be found by Mathematica. However the resulting formula
is extremely long, therefore we only evaluate the integral for $\alpha = 0.5,\ 1.5$. We evaluate $|I - Q_h^p|$ for different values of $h$, where $I$ denotes the true values, and perform linear regressions in log-log plots to find out the order of convergence, as shown in \cref{fig:On_diag_order}. We find that the numerical results in general match very well with theoretically predicted order of accuracy $2p+4-\alpha$ for all $p = 0,1,2$ and $\alpha = 0.5,1.5$. When the value of $p$ becomes large and $h$ becomes small, the round-off errors dominate. In this case, multi-precision will help when one chooses large $p$.

%%%%%%%%%%%%%%%%%%%%%%%%%%%%%%%%%%%%%%%%%%%%%%%%%%%%%%%%%%%%%%%%%%%%%%%%%%%%%%%%%%%%

\begin{figure}[h]
    \centering
    \subfloat[$\alpha = 0.5$]{\includegraphics[width = 6.5cm]{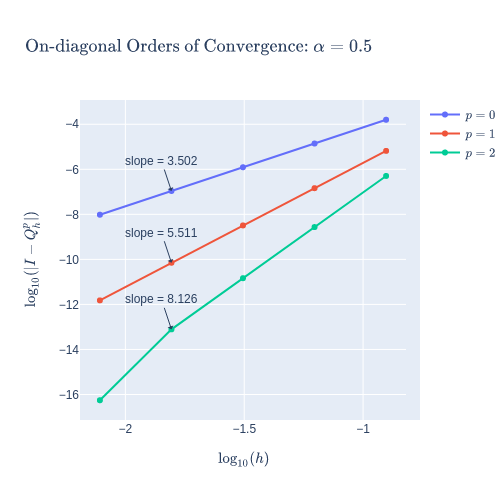}}
    \subfloat[$\alpha = 1.5$]{\includegraphics[width = 6.5cm]{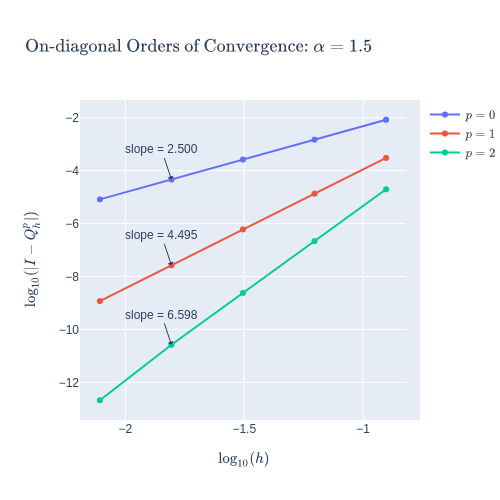}}
    \caption{Numerical orders of accuracy: Log-log plots of the quadrature error against $h$ for the modified trapezoidal rule \cref{SecondCorrectedTrapezRule} corresponding to the on-diagonal singular kernel $s(x) = \frac{x_1^2}{|x|^{2+\alpha}}$ and $p=0,1,2$, (a) $\alpha = 0.5$, (b) $\alpha = 1.5$.}
    \label{fig:On_diag_order}
\end{figure}

% \begin{table}[h!]
% \caption{On-diagonal Order of convergence for $\alpha = 0.5$}
% \begin{center}
% \begin{tabular}{ c|c|c|c } 
% \hline\hline
% Number of Layers & $p$ & Theoretical & Numerical \\
%   \hline
%   Zero & 0 & $3.5$ & $3.500336222187659$ \\
%   \hline
%   One & 1 & $5.5$ & $5.499583529152044$ \\
%   \hline
%   Two & 2 & $7.5$ & $7.548864528107378$ \\
%   \hline\hline
% \end{tabular}
% \end{center}
% \label{OrderDiagonal1}
% \end{table}

%%%%%%%%%%%%%%%%%%%%%%%%%%%%%%%%%%%%%%%%%%%%%%%%%%%%%%%%%%%%%%%%%%%%%%%%

% \begin{figure}[h]
%     \centering
%     \includegraphics[width = 10cm]{On-diag alpha 1.5.png}
%     \caption{On-diagonal with $\alpha=1.5$}
%     \label{fig2:On_diag_alpha_1.5}
% \end{figure}

% \begin{table}[h!]
% \caption{On-diagonal Order of convergence for $\alpha = 1.5$}
% \begin{center}
% \begin{tabular}{ c|c|c|c } 
% \hline\hline
% Number of Layers & $p$ & Theoretical & Numerical \\
%   \hline
%   Zero & 0 & $2.5$ & $2.500004761208461$ \\
%   \hline
%   One & 1 & $4.5$ & $4.498947458277292$ \\
%   \hline
%   Two & 2 & $6.5$ & $6.501162562846565$ \\
%   \hline\hline
% \end{tabular}
% \end{center}
% \label{OrderDiagonal2}
% \end{table}
%%%%%%%%%%%%%%%%%%%%%%%%%%%%%%%%%%%%%%%%%%%%%%%%%%%%%%%%%%%%%%%%%%%%%%%%%%%%%%%%%%%%%%%%%%%%%%%%%%%%%%%%%%%%%%%%%%%%%%%%%%%%%%%%%%%%%%%%%%%%%%%%%%%%%%%%%

\subsubsection{Off-diagonal case}
We verify the order of convergence of the corrected trapezoidal rule \cref{CorrectedTrapz_off_diag} for the off-diagonal weakly singular integral $I_{12}$ in \cref{singularint}, we choose a compact supported function \begin{align}
    \phi(x) &= (1+x_1)(1+x_2)\big((1-x_1^2)_+(1-x_2^2)_+\big)^7, 
\end{align} Note that $\phi\in C_c^6(\mathbb{R}^2)$ and $supp(\phi) = [-1,1]^2$. We use $\phi$ to test \cref{OffdiagThirdThm} for $p=1,2,3$. Note that due to limited smoothness of this $\phi$, \cref{OffdiagThirdThm} guarantees the order of accuracy of $2p+2-\alpha$ for $p=1,2$, but is inconclusive for $p=3$. However we still observe the  high order of accuracy. This implies there are room for further relaxation on smoothness criteria on $\phi$. The true value of the integral $ I = \int_{\mathbb{R}^2\setminus\{0\}}\phi(x)s(x)\ \d x$ for $\alpha = 0.5,1.5$ can be found by Mathematica. We evaluate $|I - Q_h^p|$ for different values of $h$, where $I$ denotes the true values and perform linear regressions in log-log plots to find out the order of convergence, as shown in \cref{fig:Off_diag_order}. We find that the numerical results in general match very well with the theoretically predicted order of accuracy $2p+2-\alpha$ for $p = 1,2,3$ and $\alpha = 0.5,1.5$.

% \begin{figure}[h]
%     \centering
%     \includegraphics[width = 10cm]{Off-diag alpha 0.5.png}
%     \caption{Off-diagonal with $\alpha=0.5$}
%     \label{fig:Off_diag_alpha_0.5}
% \end{figure}

\begin{figure}[h]
    \centering
    \subfloat[$\alpha = 0.5$]{\includegraphics[width = 6.5cm]{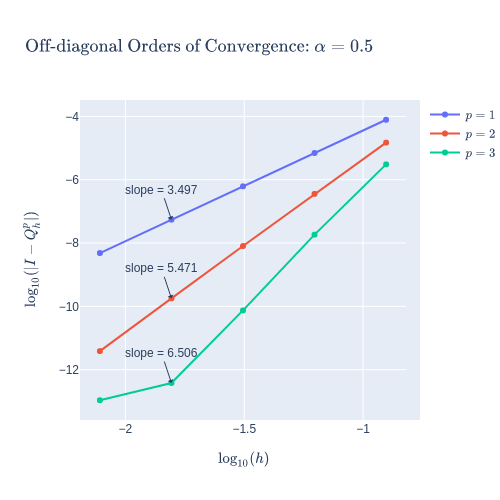}}
    \subfloat[$\alpha = 1.5$]{\includegraphics[width = 6.5cm]{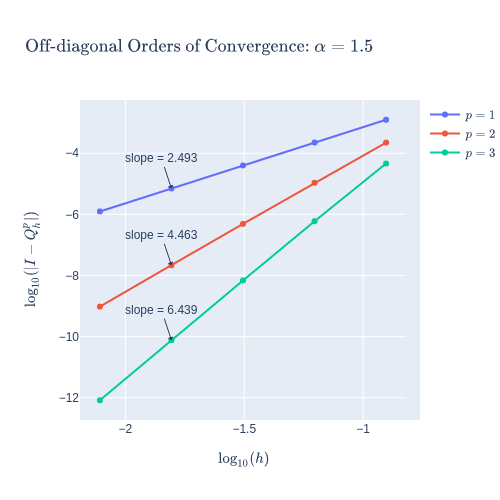}}
    \caption{Numerical orders of accuracy: Log-log plots of the quadrature error against $h$ for the modified trapezoidal rule \cref{CorrectedTrapz_off_diag} corresponding to the off-diagonal singular kernel $s(x) = \frac{x_1x_2}{|x|^{2+\alpha}}$ and $p=1,2,3$, (a) $\alpha = 0.5$, (b) $\alpha = 1.5$.}
    \label{fig:Off_diag_order}
\end{figure}

% \begin{figure}[h]
%     \centering
%     \includegraphics[width = 10cm]{Off-diag alpha 1.5.png}
%     \caption{Off-diagonal with $\alpha=1.5$}
%     \label{fig:Off_diag_alpha_1.5}
% \end{figure}
% \begin{table}
% \caption{Order of Convergence of Off-diagonal Corrected Trapezoidal Rule for Different Layers when $\alpha = 0.5$}
% \begin{center}
% \begin{tabular}{ c|c|c|c } 
% \hline\hline
% Number of Layers & $p$ & Theoretical & Numerical \\
% \hline
%     None & $0$\text{ or }$1$ & $3.5$ & $3.499387424874827$ \\
%   \hline
%   One & 2 & $5.5$ & $5.492864158415529$ \\
%   \hline
%   Two & 3 & $7.5$ & $7.444233929004576$ \\
%   \hline\hline
% \end{tabular}
% \end{center}
% \label{OrderOff_0.5}
% \end{table}

% \begin{table}
% \caption{Order of Convergence of Off-diagonal Corrected Trapezoidal Rule for Different Layers when $\alpha = 1.5$}
% \begin{center}
% \begin{tabular}{ c|c|c|c } 
% \hline\hline
% Number of Layers & $p$ & Theoretical & Numerical \\
% \hline
%     None & $0$\text{ or }$1$ & $2.5$ & $2.498500238742231$ \\
%   \hline
%   One & 2 & $4.5$ & $4.491787622467762$ \\
%   \hline
%   Two & 3 & $6.5$ & $6.486266399942143$ \\
%   \hline\hline
% \end{tabular}
% \end{center}
% \label{OrderOff_1.5}
% \end{table}

\section{Conclusion}
We proposed an arbitrarily high-order modified trapezoidal rules for both on-diagonal and off-diagonal weakly singular integrals $I_{i,j} = \int_{\mathbb{R}^2}\phi(x)\frac{x_ix_j}{|x|^{2+\alpha}}\d x$, $i,j\in\{1,2\}$ with some sufficient smooth function $\phi$ with compact support. The quadrature rule is a punctured-hole trapezoidal rule with correction terms. We proved the order of accuracy of the modified quadrature given the number of correction layers. We focused on error due to singularity in this paper. The quadrature rule can be combined with any boundary error correction for regular functions $\phi$ without compact support to attain high-order convergence. The correction weights can be pre-computed and stored for future use. We tabulated correction weights in more-than-17 digits in scientific form for $\alpha = 0.5, 1.5$ and for both on-diagonal and off-diagonal case. We provided theoretical guarantee that the quadrature rule works with arbitrary $p$, the number associated with correction layers, though we do not recommend to use large $p$ since numerical evidence suggested that linear systems \cref{MatrixLinearSystem} and \cref{LinearSysOff_Second} become increasingly ill-conditioned for large $p$. By using the method suggested in \cref{Rmk_on_MarinTornburg_prob}, one can prove the non-singularity left in the work of \cite{MarinTornberg2014}. For future works, one can further relax the smoothness criteria for $\phi$ in \cref{SecondMainThm,,OffdiagThirdThm}. Alternatively, the modified trapezoidal rule is of course not confined to singular kernels described here. With proper modification, one can easily adapt this rule to other singular kernels in 2D such as $x_1\log{(|x|)}$ or $\frac{P(x_1,x_2)}{|x|^{\sfrac{(2k+1)}{2}}}$ where $P(x_1,x_2)$ is a irreducible polynomial in $x_1,x_2$ with degree $k$. Another possible direction is to generalize the quadrature rule to arbitrary $n$ dimensions, as some results in this paper are already set up in $n$ dimensions.

\bibliographystyle{plain}
\bibliography{main}

\end{document}